\documentclass{amsart}
 
\usepackage{mystyle}

\title{On the periodic $v_2$-self-map of $A_1$}
\author[Prasit B.]{Prasit Bhattacharya$^{1,*}$}
\address{$^1$Department of Mathematics, University of Notre Dame, 106 Hayes-Healy Hall, Notre Dame, IN 46556, USA}
\address{$^1$Tel: +1(574) 631-7776 }
\address{$^*$Corresponding author}
\email{$^1$prasbhat@indiana.edu}
\author[P. Egger]{Philip Egger$^2$}
\address{$^{2,3}$Department of Mathematics, Northwestern University, 2033 Sheridan Road, Evanston, IL 60208, USA}
\address{$^2$Tel: +1(847)467-1958}
\email{$^2$philip.egger@math.northwestern.edu}
\author[M. Mahowald]{Mark Mahowald$^3$}
\thanks{Prasit Bhattacharya is supported in part by NSF through grant DMS-1105255.}

\begin{document}

\maketitle

\begin{center}This paper is dedicated to the memory of Mark Mahowald (1931-2013).\end{center}

\begin{abstract} 
We prove that the minimal $v_2$-self-map of the $2$-local spectrum $A_1$ has periodicity $32$.
\end{abstract}

\begin{center}Keywords: stable homotopy, $v_2$-periodicity\end{center}

\section*{Acknowledgments}
The first and second authors would like to thank  Mark Behrens, Bob Bruner, Paul Goerss, Mike Hill and Mike Mandell for their invaluable assistance and encouragement throughout this project, as well as Irina Bobkova for some helpful discussions. 
 \newline

\begin{conv*}Throughout this paper we work in the stable homotopy category of spectra localized at the prime $2$.

\end{conv*}
\section{Introduction}
Let $K(n)$ be the $n^{th}$ Morava $K$-theory. Let $\mathcal{C}_0$ be the category of $2$-local finite spectra, $\mathcal{C}_n \subset\mathcal{C}_0$ be the full subcategory of $K(n-1)$-acyclics and $\mathcal{C}_{\infty}$ be the full subcategory of contractible spectra. Hopkins and Smith \cite{HS} showed that the $\mathcal{C}_n$ are thick subcategories of $\mathcal{C}_0$ (in fact, they are the only thick subcategories of $\mathcal{C}_0$) and they fit into a sequence
\[ \mathcal{C}_{0} \supset \mathcal{C}_{1} \supset \ldots \supset \mathcal{C}_n \supset \ldots \supset \mathcal{C}_{\infty}.\] 
We say a finite spectrum $X$ is of type $n$ if $X \in \mathcal{C}_n \setminus\mathcal{C}_{n+1}$.

A self-map $v:\Sigma^k X\rightarrow X$ of a finite spectrum $X$ is called a \emph{$v_n$-self-map} if 
\[K(n)_*(v): K(n)_*(X) \to K(n)_*(X) \]
 is an isomorphism. For a finite spectrum $X$, a self-map $v: \Sigma^k X \to X$ can also be regarded as an element of $\pi_k(X \sma DX)$, where $DX$ is the Spanier-Whitehead dual of $X$. 

For any ring spectrum $E$, let $H_{E}$ denote the $E$-Hurewicz natural transformation \[ H_{E}: \pi_*(\underline{ \ \ }) \to E_*( \underline{ \ \ }).\] 
Let $k(n)$ denote the connective cover of $K(n)$. If $v: S^k \to X \sma DX$ is a $v_n$-self-map then $H_{k(n)}(v)\in k(n)_*(X\sma DX)$ has to be the image of $v_n^{m} \in k(n)_* \iso \Ft[v_n]$, for some positive integer $m$, under the map 
\[ k(n)_* \unit:k(n)_* \to k(n)_*(X \sma DX),\]
where $\unit: S^0 \to X \sma DX$ is the unit map. The value $m$ is called the \emph{periodicity} of the $v_n$-self-map $v$. We call $v$ a \emph{minimal $v_n$-self-map} for $X$, if $v$ is a $v_n$-self-map with smallest periodicity. An easy consequence of \cite[Theorem~9]{HS} is that the periodicity of a minimal $v_n$-self-map is always a power of $2$.

Hopkins and Smith showed, among other things, that every type $n$ spectrum admits a $v_n$-self-map and the cofiber of a $v_n$-self-map is of type $n+1$. However, not much is known about the minimal periodicity of such $v_n$-self-maps.

The sphere spectrum $S^0$ is a type $0$ spectrum with a $v_0$-self-map $2:S^0\rightarrow S^0$. The cofiber of this $v_0$-self-map is the type $1$ spectrum $M(1)$. The spectrum $M(1)$ is known to admit a unique minimal $v_1$-self-map of periodicity $4$. The cofiber of this $v_1$-self-map is denoted by $M(1,4)$. In 2008, Behrens, Hill, Hopkins and the third author \cite{BHHM} showed that the minimal $v_2$-self-map on $M(1,4)$ is $v:\Sigma^{192}M(1,4)\rightarrow M(1,4)$, which has periodicity $32$. 


Instead of $S^0$, we can start with the type $0$ spectrum $C\eta$, the cofiber of \linebreak $\eta:S^1\rightarrow S^0$. The spectrum $C\eta$ admits a non-zero $v_0$-self-map $2\sma 1_{C\eta}:C\eta\rightarrow C\eta$, with cofiber $M(1) \sma C\eta :=Y$. The type $1$ spectrum $Y$ admits eight minimal $v_1$-self-maps of periodicity $1$. These eight maps are constructed in \cite{DM81} using stunted projective spaces. The cofiber of any of the $v_1$-self-maps is referred to as $\A$. Though there are eight different $v_1$-self-maps, there are only four different homotopy types of the cofibers $\A$ (see \cite[Proposition~2.1]{DM81}). 

Let $A(1)$ be the subalgebra of the Steenrod algebra $A$ generated by $Sq^1$ and $Sq^2$. It turns out that the cohomology of any homotopy type of $\A$ is a free $A(1)$-module on one generator. However, different homotopy types of $\A$ have different $A$-module structures, which are distinguished by the action of $Sq^4$. We depict the cohomologies of the four different spectra $\A$ in Figure~\ref{fig:picdeg} where the red square brackets represent an action of $Sq^4$, the blue curved lines represent an action of $Sq^2$, and the black straight lines represent an action of $Sq^1$. The subalgebra $A(1)$ has four different $A$-module structures each of which corresponds to a homotopy type of $\A$. Any $A$-module structure on $A(1)$ has a nontrivial $Sq^4$ action on the generator in degree $1$ forced by the Adem relations. However, there are choices for $Sq^4$ actions to be trivial or nontrivial on generators in degree $0$ and degree $2$, thus giving us four different $A$-module structures. We denote different homotopy types of $\A$ using the notation $\A[ij]$ where $i$ and $j$ are the indicator functions for the action of $Sq^4$ on the generators in degree $0$ and degree $2$ respectively. The spectra $\A[01]$ and $\A[10]$ are self-dual, i.e. $\A[01] = \Sigma^6D\A[01]$ and $\A[10] = \Sigma^6D\A[10]$, whereas $\A[00]$ and $\A[11]$ are dual to each other, i.e. $\A[00] = \Sigma^{6}D\A[11]$. This is a consequence of the fact that 
\[\chi(Sq^4) = Sq^4 + Sq^3Sq^1,\]
 where $\chi:A \to A$ is the canonical antiautomorphism of the Steenrod algebra.   
\begin{figure}[h] 
\centering
\begin{tikzpicture}[scale = .7]
\node (a0) at (2,0) {$\bullet$};
\node (a1) at (2,1) {$\bullet$};
\node (a2) at (2,2) {$\bullet$};
\node (a3) at (2,3) {$\bullet$};
\node (a7) at (3,6) {$\bullet$};
\node (a6) at (3,5) {$\bullet$};
\node (a5) at (3,4) {$\bullet$};
\node (a4) at (3,3) {$\bullet$};
\draw[-] (a0) -- (a1);
\draw[-] (a2) -- (a3);
\draw[-] (a4) -- (a5);
\draw[-] (a6) -- (a7);
\draw[blue, bend left] (a0) to (a2);
\draw[blue, bend right] (a1) to (a4);
\draw[blue, bend left] (a3) to (a6);
\draw[blue, bend right] (a5) to (a7);
\draw[blue, out= 20, in = -160] (a2) to (a5);
\draw[red] (a1) -- (1.5,1) -- (1.5,5) -- (a6);
\end{tikzpicture} \hspace{25pt}
\begin{tikzpicture}[scale = .7]
\node (a0) at (2,0) {$\bullet$};
\node (a1) at (2,1) {$\bullet$};
\node (a2) at (2,2) {$\bullet$};
\node (a3) at (2,3) {$\bullet$};
\node (a7) at (3,6) {$\bullet$};
\node (a6) at (3,5) {$\bullet$};
\node (a5) at (3,4) {$\bullet$};
\node (a4) at (3,3) {$\bullet$};
\draw[-] (a0) -- (a1);
\draw[-] (a2) -- (a3);
\draw[-] (a4) -- (a5);
\draw[-] (a6) -- (a7);
\draw[blue, bend left] (a0) to (a2);
\draw[blue, bend right] (a1) to (a4);
\draw[blue, bend left] (a3) to (a6);
\draw[blue, bend right] (a5) to (a7);
\draw[blue, out= 20, in = -160] (a2) to (a5);
\draw[red] (a1) -- (1.5,1) -- (1.5,5) -- (a6);
\draw[red] (a0) --(3.5, 0) -- (3.5, 4) -- (a5);
\end{tikzpicture}\hspace{30pt}
\begin{tikzpicture}[scale = .7]
\node (a0) at (2,0) {$\bullet$};
\node (a1) at (2,1) {$\bullet$};
\node (a2) at (2,2) {$\bullet$};
\node (a3) at (2,3) {$\bullet$};
\node (a7) at (3,6) {$\bullet$};
\node (a6) at (3,5) {$\bullet$};
\node (a5) at (3,4) {$\bullet$};
\node (a4) at (3,3) {$\bullet$};
\draw[-] (a0) -- (a1);
\draw[-] (a2) -- (a3);
\draw[-] (a4) -- (a5);
\draw[-] (a6) -- (a7);
\draw[blue, bend left] (a0) to (a2);
\draw[blue, bend right] (a1) to (a4);
\draw[blue, bend left] (a3) to (a6);
\draw[blue, bend right] (a5) to (a7);
\draw[blue, out= 20, in = -160] (a2) to (a5);
\draw[red] (a1) -- (1.5,1) -- (1.5,5) -- (a6);
\draw[red] (a2) -- (1.3, 2) -- (1.3, 6) -- (a7);
\end{tikzpicture}\hspace{30pt}
\begin{tikzpicture}[scale = .7]
\node (a0) at (2,0) {$\bullet$};
\node (a1) at (2,1) {$\bullet$};
\node (a2) at (2,2) {$\bullet$};
\node (a3) at (2,3) {$\bullet$};
\node (a7) at (3,6) {$\bullet$};
\node (a6) at (3,5) {$\bullet$};
\node (a5) at (3,4) {$\bullet$};
\node (a4) at (3,3) {$\bullet$};
\draw[-] (a0) -- (a1);
\draw[-] (a2) -- (a3);
\draw[-] (a4) -- (a5);
\draw[-] (a6) -- (a7);
\draw[blue, bend left] (a0) to (a2);
\draw[blue, bend right] (a1) to (a4);
\draw[blue, bend left] (a3) to (a6); 
\draw[blue, bend right] (a5) to (a7);
\draw[blue, out= 20, in = -160] (a2) to (a5);
\draw[red] (a1) -- (1.5,1) -- (1.5,5) -- (a6);
\draw[red] (a2) -- (1.3, 2) -- (1.3, 6) -- (a7);
\draw[red] (a0) --(3.5, 0) -- (3.5, 4) -- (a5);
\end{tikzpicture}
\caption{The $A$-module structures of $H^*(\A[00])$, $H^*(\A[10])$, $H^*(\A[01])$ and $H^*(\A[11])$.}
\label{fig:picdeg}
\end{figure}
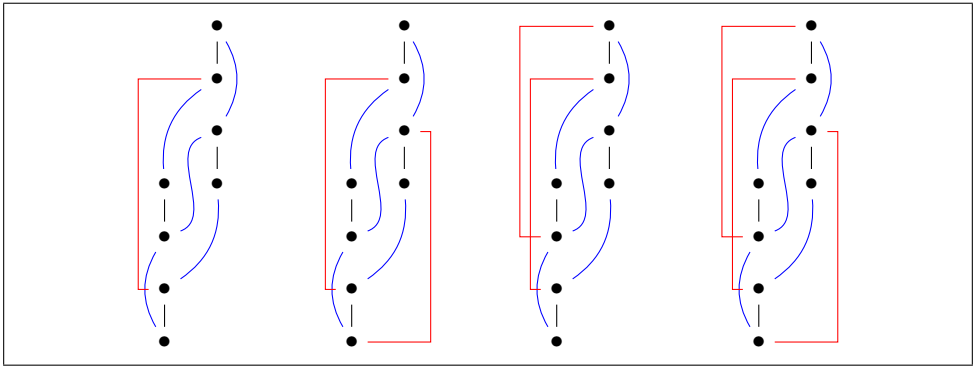

It is worth noting that $\A$ is created in a way similar to $M(1,4)$, where $C\eta$ is analogous to $S^0$, and $Y$ is analogous to $M(1)$. Therefore, it is reasonable to ask whether $\A$ has the same $v_2$-periodicity as $M(1,4)$. The minimal $v_1$-self-map of $Y$ has periodicity $1$, which is less than the periodicity of the minimal $v_1$-self-map on $M(1)$, which is $4$. Hence, it is natural to ask if any of the four models of $\A$ admit a $v_2$-self-map of periodicity $2^k$, where $k \leq 4$. In \cite{BHHM}, the third author conjectured that the minimal $v_2$-self-map of $\A$ should have periodicity $32$. The goal of this paper is to prove the following~\footnote{In \cite{DM81}, Davis and the third author claimed, incorrectly, that the periodicity of minimal $v_2$-self-maps on $M(1,4)$ and the two self-dual models of $\A$, namely $\A[01]$ and $\A[10]$, as $8$. After successfully correcting the $v_2$-periodicity of $M(1,4)$ in \cite{BHHM}, the $v_2$-periodicity of $\A$ was called into question by the third author.}:
\begin{main}\label{mainthm}For all four models of $\A$, the minimal $v_2$-self-map
\[v:\Sigma^{192}\A\rightarrow \A\]
has periodicity $32$.
\end{main} 
\begin{notn} \label{subsec:not}For any ring spectrum $E$, $\unit_{E}: S^0 \to E$ will denote the unit map. The unit map $\unit_E$ induces the the Hurewicz natural transformation
\[ H_{E}: \pi_*(\underline{ \ \ }) \to E_*( \underline{ \ \ })\]
as introduced earlier. When $E = \A \sma D\A$, we simply use $\unit: S^0 \to \A \sma D\A$ to denote the unit map. Let $\inc:S^0 \hookrightarrow \A$ be the map that represents the inclusion of the bottom cell. Let $j: \A \sma D\A \to \A$ denote the map $1_{\A} \sma D\inc$.  
\end{notn}
\begin{notn}
To lighten the notations, we use $Ext_{T}^{s,t}(X)$ to denote \linebreak $Ext_{T}^{s,t}(H^{*}(X), \Ft)$, where $T$ is a subalgebra of the Steenrod algebra $A$.    
\end{notn}


\subsection{Outline} To prove Main~Theorem~\ref{mainthm}, we use the fact that the spectrum $tmf$ detects certain $v_2$-periodic elements. More specifically, the unit map $\unit_{k(2)}:S^0\to k(2)$ factors through $tmf$, i.e. we have\[\unit_{k(2)}:S^0\overset{\unit_{tmf}}{\to}tmf\overset{r}{\to}k(2).\]The induced map in homotopy
\[r_*:tmf_* \to k(2)_*\]
sends $\Delta^{8}$, the periodicity generator of $tmf_*$, to $v_2^{32}$. Since $\A$ is a type $2$ spectrum, we know that $\Delta^8$ has a nonzero image under the composition 
\[ tmf_* \overset{r_*}{\longrightarrow} k(2)_* \overset{k(2)_*\unit}{\longrightarrow} k(2)_*(\A \sma D\A).\]
Therefore, from the commutative diagram 
\[ \xymatrix{
tmf_* \ar[rr]^-{tmf_*\unit } \ar[d]_{r_*} && tmf_*(\A \sma D\A) \ar[d]^{r_*(\A \sma D\A)} \\
k(2)_* \ar[rr]_-{k(2)_*\unit  }  && k(2)_*(\A \sma D\A)
}\] 
we see that $k(2)_*\unit(v_2^{32})$ lifts to $tmf_*(\A \sma D\A)$. We can choose the lift to be $tmf_*\unit(\Delta^8)$. This does not eliminate the possibility that smaller powers of $k(2)_*\unit(v_2)$ could lift to $tmf_*(\A \sma D\A)$. However, if $k(2)_*\unit(v_2^8)$ and $k(2)_*\unit(v_2^{16})$ do not lift to $tmf_*(\A \sma D\A)$, then they will not lift to $\pi_*(\A \sma D\A)$. So we analyse the map of Adams spectral sequences induced by $r:tmf \to k(2)$. 

It is well-known that $H^{*}(tmf)$ as an $A$-module is isomorphic to $A\modmod A(2)$, where $A(2)$ is the subalgebra of $A$ generated by $Sq^1, Sq^2$ and $Sq^4$. Therefore, applying a change of rings formula, we see that $Ext_{A(2)}^{s,t}(X)$ is the $E_2$ page of the Adams spectral sequence  
\[ E_2^{s,t} = Ext_{A(2)}^{s,t}(X) \Rightarrow tmf_{t-s}(X).\] 
Similarly, we have an Adams spectral sequence 
\[E_2^{s,t} = Ext_{E(Q_2)}^{s,t}(X) \Rightarrow k(2)_{t-s}(X), \]
which is a manifestation of the fact that $H^*(k(2)) = A\modmod E(Q_2)$.

The map $\unit: S^0 \to \A \sma D\A$ induces the following commutative diagram of spectral sequences 
\begin{equation}\label{maindiagram}\xymatrix{
Ext_{A(2)}^{s,t}(S^0)\ar[d]\ar[r]^-{\unit^{tmf}_*}&Ext_{A(2)}^{s,t}(\A\sma D\A)\ar[d]\\
Ext_{E(Q_2)}^{s,t}(S^0)\ar[r]^-{\unit^{k(2)}_*}&Ext_{E(Q_2)}^{s,t}(\A\sma D\A).
}\end{equation}
It is well known that \[v_2^8 \in Ext_{E(Q_2)}^{8,48+8}(S^0)\]is the image of the nonnilpotent element $b_{3,0}^4 \in Ext_{A(2)}^{8,48+8}(S^0)$ (see \cite{Bau, Hen}). Since $\A$ is a type $2$ spectrum, the element $\unit^{k(2)}_*(v_2^8) \in Ext_{E(Q_2)}^{8,48+8}(\A\sma D\A)$ is nonnilpotent. Consequently, \[\unit^{tmf}_*(b_{3,0}^4) \in Ext_{A(2)}^{8,48+8}(\A\sma D\A)\]is nonnilpotent. Thus, $\unit^{k(2)}_*(v_2^{8n})$ lifts to a nonzero element of $Ext_{A(2)}^{8n, 48n +8n}(\A \sma D\A)$ for every $n \in \mathbb{N}$, which can be chosen to be $\unit^{tmf}_*(b_{3,0}^{4n})$.

In Section~\ref{one}, we warm up by computing $Ext_{A(2)}^{s,t}(\A)$ using the May spectral sequence and compute its vanishing line for later use. In Section~\ref{two} we show that $\unit^{tmf}_*(b_{3,0}^4)$ admits a $d_2$ differential and $\unit^{tmf}_*(b_{3,0}^8)$ admits a $d_3$ differential in the Adams spectral sequence
\[ E_2^{s,t} = Ext^{s,t}_{A(2)}(\A \sma D \A) \To tmf_{t-s}(\A \sma D\A).\]
This will imply the nonexistence of a $8$-periodic or $16$-periodic $v_2$-self-map of $\A$. We will recall the algebraic $tmf$ resolution of \cite{BHHM} and use the resulting spectral sequence to show that for every $n \in \mathbb{N}$, the element $\unit^{tmf}_*(b_{3,0}^{4n})$ lifts to $Ext_A^{8n,48n + 8n}(\A \sma D\A)$ under the map induced by $H_{tmf}$. Furthermore, we show that the lifts of $\unit^{tmf}_*(b_{3,0}^4)$ and $\unit^{tmf}_*(b_{3,0}^8)$ support a $d_2$ and a $d_3$ differential respectively in the Adams spectral sequence 
\[E_2^{s,t} = Ext^{s,t}_{A}(\A \sma D\A) \To \pi_{t-s}(\A \sma D\A).\]
This extra effort enables us to identify some $d_2$ and $d_3$ differentials in the above spectral sequence, which will play a crucial role in the proof of the existence of a $32$-periodic $v_2$-self-map of $\A$. Thus, the existence of a $32$-periodic $v_2$-self-map of $\A$ boils down to showing that the lift of $\unit^{tmf}_*(b_{3,0}^{16})$, which we'll call $\sv$, is a permanent cycle in the Adams spectral sequence 
\[E_2^{s,t} = Ext^{s,t}_{A}(\A \sma D\A) \To \pi_{t-s}(\A \sma D\A).\]

Note that $\sv$ cannot be a target of a differential as its image in $Ext_{E(Q_2)}^{32,192+32}(\A \sma D\A)$ is not a target of a differential. Further, $\sv$ cannot support a nontrivial $d_2$ or $d_3$ differential by the Leibniz rule. In Section~\ref{four} we use all prior knowledge of $d_2$ and $d_3$ differentials, including an important $d_3$ differential found in Section~\ref{three}, to show that the potential targets of $d_r$ differentials for $r \geq 4$ are either zero or not present in the Adams $E_4$ page. This will conclude the proof of Main~Theorem~\ref{mainthm}.

\begin{notn}
For the rest of the paper, we will abusively denote any \linebreak $x \in Ext_{A(2)}^{s,t}(S^0)$ and $ \unit^{tmf}_*(x) \in Ext_{A(2)}^{s,t}(\A \sma D\A)$ and  sometimes their lifts in $Ext_{A}^{s,t}(S^0)$ and $Ext_{A}^{s,t}(\A \sma D\A)$ respectively under $H_{tmf*}$, just by $x$. This will allow us to suppress cumbersome notations. We will make sure that the ambient group in which $x$ belongs is clear from the context.     
\end{notn}

\subsection{Use of Bruner's Ext software}
We will use this software (see Appendix~\ref{appendix} or \cite{Bru} for a description of the program) for two purposes. Given any $A(2)$-module $M$, finitely generated as an $\Ft$-vector space, the program can compute the groups $Ext_{A(2)}^{s,t}(M, \Ft)$ to the extent of identifying generators in each bidegree within a finite range, determined by the user. Since we are interested in $Ext_{A(2)}^{s,t}(X)$ for finite spectra $X$, such as $\A\sma D\A$, whose cohomology structures as $A(2)$-modules are known, this suits our task perfectly. The second purpose is the following: As any finite spectrum $X$ is an $S^{0}$-module, $Ext_{A(2)}^{*,*}(X)$ is a module over $Ext_{A(2)}^{*,*}(S^{0})$. Given an element $x \in Ext_{A(2)}^{s,t}(X)$, the action of $Ext_{A(2)}^{*,*}(S^{0})$ can be computed using the \texttt{dolifts} functionality of the software. Summary of the output of the Bruner's program that is needed for some of the results in Section~\ref{three} and Section~\ref{four} are listed in Appendix~\ref{appendixB} and Appendix~\ref{appendixC} respectively.

One should also be aware that Main~Theorem~\ref{mainthm} is by no means a consequence of the programming output. However, parts of the proof are reduced to pure algebraic computation, which can be performed using Bruner's program.

\section{Computation of $Ext_{A(2)}^{s,t}(A_1)$ and its vanishing line}\label{one}

J.P. May in his thesis~\cite{May} introduced a filtration of the Steenrod algebra called the May filtration, which induces a filtration of the cobar complex $C(\Ft,A_*, \Ft)$. This filtration gives a trigraded spectral sequence
\[ E_1^{s,t,u} = \Ft[h_{i,j}: i\geq 1, j \geq 0] \Rightarrow Ext_{A}^{s,t}(S^0),|h_{i,j}|=(1,2^j(2^i-1),2i-1),\]
with differentials $d_r$ of tridegree $(1, 0, 1-2r)$,     which converges to the $E_2$ page of the Adams spectral sequence
\[Ext_{A}^{s,t}(S^0) \Rightarrow \pi_{t-s}(S^0).\]
The element $h_{i,j}$ corresponds to the class $[\xi_i^{2^j}]$ in the cobar complex $C(\Ft, A_*, \Ft)$. We stick to the notation introduced by Tangora in his thesis~\cite{Tan}. For example, $h_{1,j}$ is abbreviated by $h_j$. Meanwhile, there are many elements $h_{i,j}$ that are not $d_1$-cycles in the May spectral sequence, however, even in these cases, the Leibniz rule means that $h_{i,j}^2$ will be $d_1$-cycles. To get around the awkwardness of talking about $h_{i,j}^2$ in later pages of the May spectral sequence, where $h_{i,j}$ may not even exist, Tangora uses $b_{i,j}$ to denote $h_{i,j}^2$ from the $E_2$ page onwards.

One can use the same May filtration on the subalgebra $A(2)$ of $A$, to obtain a filtration on the cobar complex $C(\Ft, A(2)_*, \Ft)$. Thus we get a May spectral sequence with finitely many differentials
\[ \Ft[h_0,h_1,h_2,h_{2,0},h_{2,1},h_{3,0}] \Rightarrow Ext_{A(2)}^{s,t}(S^0)\]
all of which have been computed (see~\cite{DMext}). The bigraded ring $Ext_{A(2)}^{s,t}(S^0)$ is the Adams $E_2$ page for the homotopy groups of $tmf$. 

 We have obtained $\A$ by a series of cofibrations, 
\[ S^{1} \overset{\eta}{\to} S^{0} \to C\eta\]
\[ C\eta \overset{2}{\to} C\eta \to Y \]
and
\[ \Sigma^2Y \overset{v_1}{\to} Y \to\A.\]
The maps $2$, $\eta$ and $v_1$ are detected by $h_0$, $h_1$ and $h_{2,0}$, respectively, in the May spectral sequence. Using the fact that cofiber sequences induce long exact sequences of $E_1$ pages of the May spectral sequence, we get that the $E_1$ page of the May spectral sequence converging to 
$Ext_{A(2)}^{s,t}(\A)$ is 
\[ \Ft[h_2, h_{2,1}, h_{3,0}]\To Ext_{A(2)}^{s,t}(\A).\]
Alternatively, using a change of rings formula, we see that the cobar complex (whose cohomology is $Ext_{A(2)}^{s,t}(\A)$) is
\[C(\Ft, A(2)_*, A(1)_*)\iso C(\Ft,(A(2)\modmod A(1))_*,\Ft),\]
 hence a quotient of $C(\Ft, A(2)_*, \Ft)$. Thus, the filtration on $C(\Ft, A(2)_*, \Ft)$ induces a filtration on $C(\Ft, A(2)_*, A(1)_*)$ as a result of which $\Ft[h_2, h_{2,1}, h_{3,0}]$ is a module over $\Ft[h_0,h_1,h_2,h_{2,0},h_{2,1},h_{3,0}]$.

The $d_1$ differentials in the May spectral sequence 
\[ \Ft[h_0,h_1,h_2,h_{2,0},h_{2,1},h_{3,0}] \Rightarrow Ext_{A(2)}^{s,t}(S^0)\]
come from the coproduct on $A(2)_*$. It is well known that $d_1(h_2) =0$, $d_1(h_{2,1}) = h_1 h_{2}$ and $d_1(h_{3,0})=h_0h_{2,1} + h_2h_{2,0}$. Under the quotient map 
\[ \Ft[h_0,h_1,h_2,h_{2,0},h_{2,1},h_{3,0}] \twoheadrightarrow \Ft[h_2,h_{2,1},h_{3,0}] \]
all  the images of the above differentials map to zero. Therefore, there are no $d_1$ differentials in the May spectral sequence 
\[\Ft[h_2, h_{2,0}, h_{3,0}] \Rightarrow Ext_{A(2)}(A_1). \]  
One can use Nakamura's formula to compute higher May differentials. The operations $Sq_i$ on the cobar complex of $C(\Ft, A_*, \Ft)$, defined by $Sq_i(x)=x\cup_ix+\delta x\cup_{i+1}x$ (see \cite{Nak}), satisfy
\begin{eqnarray*}
Sq_0(h_{i,j})&=&h_{i,j}^2 \\
Sq_0(b_{i,j})&=& b_{i,j}^2 \\ 
Sq_1(h_{i,j})&=& h_{i,j+1}
\end{eqnarray*}
as well as Cartan's formulas (see~\cite[Proposition~4.4 and Proposition~4.5]{Nak})
\begin{eqnarray*}
Sq_0(xy)&=&Sq_0(x)Sq_0(y)\\
Sq_1(xy) &=& Sq_1(x)Sq_0(y) + Sq_0(x)Sq_1(y)
\end{eqnarray*}
whenever $x$ and $y$ are represented by elements in appropriate pages of the May spectral sequence.
In particular we have
\[Sq_1(x^2)=0\]
for every $x$.
The differential $\delta$ in the cobar complex $C(\Ft, A_*, \Ft)$, satisfies the relation 
\begin{equation} \label{eqn:nakformula}
\delta Sq_i = Sq_{i+1} \delta
\end{equation}
for $i \geq 0$ (see \cite[Lemma~4.1]{Nak}) and is often called Nakamura's formula in the literature.

Since the May spectral sequence is obtained by filtering the cobar complex, the above formula helps in detecting differentials in the May spectral sequence. Since the cobar complex \[C(\Ft, A(2)_*, A(1)_*)\iso C(\Ft,(A(2)\modmod A(1))_*,\Ft),\] is a quotient of $C(\Ft, A(2)_*, \Ft)$, we apply~\eqref{eqn:nakformula} to find differentials in the May spectral sequence for $A_1$.
\begin{lem}\label{MSS:ExtA2A1}
In the May spectral sequence\[\Ft[h_2, h_{2,1}, h_{3,0}]\Rightarrow Ext_{A(2)}^{s,t}(\A),\]we have 
\begin{itemize}
\item $d_2( b_{2,1}) = h_2^3$
\item $d_3( b_{3,0}) =  h_2^{2}h_{2,1}$
\item $d_4( b_{3,0}^2) = h_2b_{2,1}^2$
\end{itemize}
and the spectral sequence collapses at $E_5$.
\end{lem}
\begin{proof}
In the May spectral sequence 
\begin{equation} \label{ss:tmf}
\Ft[h_0,h_1,h_2,h_{2,0},h_{2,1},h_{3,0}] \Rightarrow Ext_{A(2)}^{s,t}(S^0) 
\end{equation}
the differentials $d_2( b_{2,1}) = h_2^3$ and $d_4( b_{3,0}^2) = h_2b_{2,1}^2$ translate into differentials in $Ext_{A(2)}(A_1)$. In the cobar complex, $b_{3,0}$ is represented by the element $[\xi_{3}|\xi_3]$. Since $b_{3,0} = Sq_0 h_{3,0}$, we apply~\eqref{eqn:nakformula}, to obtain
\begin{eqnarray*}
 d_3(Sq_0 h_{3,0}) &=& Sq_1 (d_1h_{3,0}) \\
 &=& Sq_1(h_0h_{2,1}+h_2h_{2,0})\\
 &=&  h_0^2h_{2,2} + h_1h_{2,1}^2 + h_2^2h_{2,1} + h_3h_{2,0}^2\\
 &=&h_2^2h_{2,1} \ \ \ \text{in the May spectral sequence for $A_1$.}
\end{eqnarray*}
Therefore, in the cobar complex $C(\Ft, A(2)_*, A(1)_*)$, it must be the case that, 
\[ \delta([\xi_3| \xi_3]) = [\xi_1^2|\xi_1^2| \xi_2^2] + \text{elements of higher May filtration.}\]
As a result we have 
\[d_3(b_{3,0})=h_2^2h_{2,1}.\]
The May spectral sequence~\ref{ss:tmf} does not have any differentials $d_r$ for $r \geq 5$, consequently no differentials in the May spectral sequence 
\[ \Ft[h_2, h_{2,1}, h_{3,0}] \Rightarrow Ext_{A(2)}^{s,t}(\A).\]
\end{proof}
\begin{figure}[h]
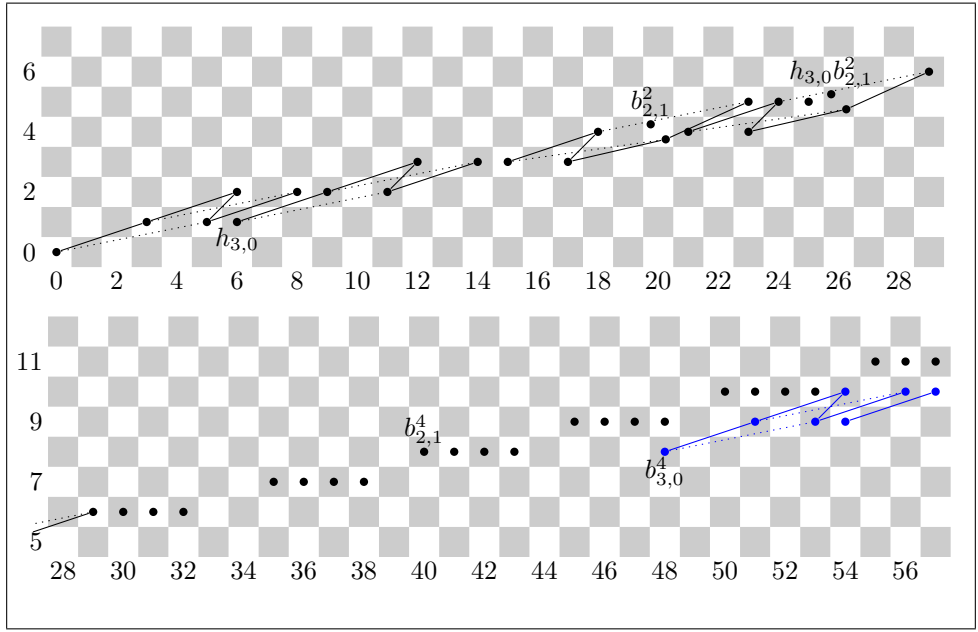

\begin{sseq}[grid= chess]{30}{8} 
\ssdropbull \ssname{g}
\onetwoH \ssname{h1} \onetwoH \ssname{h2} \oneHdiv \ssname{h3} \onetwoH \ssname{h4}
\ssgoto{g} \ssgoto{h3} \ssstroke[dotted]
\ssgoto{h1} \ssgoto{h4} \ssstroke[dotted]
\ssgoto{g} 
\kbm \ssname{k}\ssdroplabel[U]{b_{2,1}^2} \twooneH \ssgoto{k}
\kbm \ssname{k} \ssdroplabel[U]{b_{2,1}^4}\twooneH \ssgoto{k}
\ssmoveto 6 1
\ssdropbull  \ssname{v2}\ssdroplabel[D]{h_{3,0}}
\onetwoH \ssname{h1} \onetwoH \ssname{h2} \oneHdiv \ssname{h3} \onetwoH \ssname{h4}
\ssgoto{v2} \ssgoto{h3} \ssstroke[dotted]
\ssgoto{h1} \ssgoto{h4} \ssstroke[dotted]
\ssgoto{v2} 
\kbm \ssname{k}\ssdroplabel[U]{h_{3,0}b_{2,1}^2} 
\ssmoveto {15} {3}
\ssdropbull \ssname{h1} \onetwoH \ssname{h2} \oneHdiv \ssname{h3} \onetwoH \ssname{h4} \onetwoH \ssname{h5}
\ssgoto{h1} \ssgoto{h4} \ssstroke[dotted]
\ssgoto{h2} \ssgoto{h5} \ssstroke[dotted]
\ssmove {12} 2
\ssdropbull \ssname{k} \ssmove {2} 0 \ssdropbull
\ssmoveto {21} {4}
\ssdropbull \ssname{h1} \onetwoH \ssname{h2} \oneHdiv \ssname{h3} \onetwoH \ssname{h4} \onetwoH \ssname{h5}
\ssgoto{h1} \ssgoto{h4} \ssstroke[dotted]
\ssgoto{h2} \ssgoto{h5} \ssstroke[dotted]
\ssmove {12} 2
\ssdropbull \ssname{k} \ssmove {2} 0 \ssdropbull
\end{sseq}\\
\begin{sseq}[grid = chess]{28...57}{5...12}
\ssdropbull \ssname{g} 
\onetwoH \ssname{h1} \onetwoH \ssname{h2} \oneHdiv \ssname{h3} \onetwoH \ssname{h4}
\ssgoto{g} \ssgoto{h3} \ssstroke[dotted]
\ssgoto{h1} \ssgoto{h4} \ssstroke[dotted]
\ssgoto{g} 
\kbm \ssname{k}\ssdroplabel[U]{b_{2,1}^2} \twooneH \ssgoto{k}
\kbm \ssname{k} \ssdroplabel[U]{b_{2,1}^4}\twooneH \ssgoto{k}

\ssmoveto 6 1
\ssdropbull  \ssname{v2}
\onetwoH \ssname{h1} \onetwoH \ssname{h2} \oneHdiv \ssname{h3} \onetwoH \ssname{h4}
\ssgoto{v2} \ssgoto{h3} \ssstroke[dotted]
\ssgoto{h1} \ssgoto{h4} \ssstroke[dotted]
\ssgoto{v2} 
\kbm \ssname{k} \twooneH \ssgoto{k}
\kbm \ssname{k} \twooneH \ssgoto{k}
\ssmoveto {15} {3}
\ssdropbull \ssname{h1} \onetwoH \ssname{h2} \oneHdiv \ssname{h3} \onetwoH \ssname{h4} \onetwoH \ssname{h5}
\ssgoto{h1} \ssgoto{h4} \ssstroke[dotted]
\ssgoto{h2} \ssgoto{h5} \ssstroke[dotted]
\ssmove {12} 2
\ssdropbull \ssname{k} \ssmove {2} 0 \ssdropbull
\ssmoveto {21} {4}
\ssdropbull \ssname{h1} \onetwoH \ssname{h2} \oneHdiv \ssname{h3} \onetwoH \ssname{h4} \onetwoH \ssname{h5}
\ssgoto{h1} \ssgoto{h4} \ssstroke[dotted]
\ssgoto{h2} \ssgoto{h5} \ssstroke[dotted]
\ssmove {12} 2
\ssdropbull \ssname{k} \ssmove {2} 0 \ssdropbull
\ssmoveto {30} {6}
\ssdropbull \ssname{h2} \ssmove {2} 0
\ssdropbull \ssgoto{h2}
\kbm \ssmove {2} 0
\ssdropbull
\ssmoveto {36} {7}
\ssdropbull \ssname{h2} \ssmove {2} 0
\ssdropbull \ssgoto{h2}
\kbm \ssmove {2} 0
\ssdropbull
\ssmoveto {42} {8}
\ssdropbull \ssname{h2} \twooneH 
\ssmoveto {48} {9}
\ssdropbull \ssname{h2} \twooneH
\bl
\ssmoveto {48} {8}
\ssdropbull \ssname{g} \ssdroplabel[D]{b_{3,0}^4}
\onetwoH \ssname{h1} \onetwoH \ssname{h2} \oneHdiv \ssname{h3} \onetwoH \ssname{h4}
\ssgoto{g} \ssgoto{h3} \ssstroke[dotted]
\ssgoto{h1} \ssgoto{h4} \ssstroke[dotted]
\ssgoto{g} 
\kbm \ssname{k} \twooneH \ssgoto{k}
\kbm \ssname{k} \twooneH \ssgoto{k}
\ssmoveto {54} 9
\ssdropbull  \ssname{v2}
\onetwoH \ssname{h1}
\black
\ssmoveto {55} {11}
\ssdropbull \ssmove 2 0 \ssdropbull
\end{sseq}
\caption{The $E_{\infty}$-page of the May spectral sequence for $Ext_{A(2)}(A_1)$.}
\label{SSeq}
\end{figure}

In Figure~\ref{SSeq}, the solid line of slope $1$ represents multiplication by $h_1$, the solid line of slope $\frac{1}{3}$ represents multiplication by $h_2$, while the dotted line of slope $\frac{1}{5}$ represents multiplication by $h_{2,1}$. The element $b_{3,0}^4$ is the periodicity generator of $Ext_{A(2)}^{*,*}(\A)$ and the blue part is simply a repetition of the earlier black pattern. This matches the output of Bruner's program~\cite{Bru} for $Ext_{A(2)}^{s,t}(A_1)$, though different models of $\A$ may have different extensions some of which might not be detected in the May spectral sequence.   

Having computed the $E_2$ page $Ext_{A(2)}^{s,t}(A_1)$, we give a vanishing line of this spectral sequence, which will come in handy later on in the paper.  
\begin{lem} \label{vanishA2}
The group $Ext_{A(2)}^{s,t}(A_1)$ is zero if\[s>\frac{1}{5}(t-s)+1,\]and for $t-s\geq 29$, it is zero if\[s>\frac{1}{5}(t-s).\]
In other words, there is a vanishing line\[y=\frac{1}{5}x+1.\]
\end{lem}
\begin{proof}
Of the three generators of the $E_1$ page, $h_2$ has slope $\frac{1}{3}$, $h_{2,1}$ has slope $\frac{1}{5}$, and $h_{3,0}$ has slope $\frac{1}{6}$. However, while $Ext_{A(2)}^{s,t}(A_1)$ contains infinitely large powers of $h_{2,1}$ and $h_{3,0}$, it only contains powers up to 2 of $h_2$. Hence, the vanishing line of $Ext_{A(2)}^{s,t}(A_1)$ must have slope $\frac{1}{5}$, determined by $b_{2,1}^2$. Now, since $h_2b_{2,1}^2=0$, the vanishing line for stems greater than 29 is $y=\frac{1}{5}x$ and a glance at Figure~\ref{SSeq} gives us the $y$-intercept of the overall vanishing line.
\end{proof}

\section{A $d_2$ and a $d_3$ differential}\label{two}
In this section we first show that $b_{3,0}^4$ and $b_{3,0}^8$ in $Ext^{s,t}_{A(2)}(\A \sma D\A)$ support a $d_2$ and a $d_3$ differential respectively. Then we show that these differentials lift to $Ext_A^{s,t}(\A \sma D\A)$ under the map of spectral sequences induced by $H_{tmf}$.
Some of the proofs in this section as well as in the subsequent sections use Bruner's program \cite{Bru}. We provide Appendix~\ref{appendix} to help 
readers familiarize themselves with this software. 

In the Adams spectral sequence
\[ E_2^{s,t}=Ext_{A(2)}^{s,t}(S^0) \To tmf_{t-s}\]  
it is well known that $d_2(b_{3,0}^4) = e_0r$ and $d_3(b_{3,0}^8) = wgr$ (see \cite{Hen}). Using Bruner's program, we see that $e_0r$ and $wgr$ both have nonzero images in $Ext_{A(2)}^{s,t}(\A \sma D\A)$.
\begin{lem}\label{not8or16} In the Adams spectral sequence 
\[ E_2^{s,t}=Ext_{A(2)}^{s,t}(\A \sma D\A) \To tmf_{t-s}(\A\sma D\A)\]
we have $d_2(b_{3,0}^4) = e_0r$ and $d_3(b_{3,0}^8)=wgr$.
\end{lem}
\begin{proof}
In the map of Adams spectral sequences,
\[ \xymatrix{
E_2^{s,t} = Ext_{A(2)}^{s,t}(S^0) \ar[d]\ \ar@{=>}[r] & tmf_{t-s}  \ar[d] \\
E_2^{s,t} = Ext_{A(2)}^{s,t}(\A \sma D\A) \ar@{=>}[r] &  tmf_{t-s} (\A \sma D\A)
}\]
 we have established that (beware of our abusive notations as explained in Notation~\ref{subsec:not})
 \begin{eqnarray*}
 Ext_{A(2)}^{s,t}(S^0)&\overset{\unit_{*}^{tmf}}{\longrightarrow}&Ext_{A(2)}^{s,t}(\A\sma D\A)\\
 b_{3,0}^4&\mapsto&b_{3,0}^4\\
 b_{3,0}^8&\mapsto&b_{3,0}^8 \\
e_0r&\mapsto& e_0r\\
wgr & \mapsto& wgr.
 \end{eqnarray*}
Since $d_2(b_{3,0}^4) = e_0r$ in the Adams spectral sequence for $tmf_*$, it follows that we have a $d_2$-differential\[d_2(b_{3,0}^4) = e_0r.\]
As a consequence of the Leibniz rule, $d_2(b_{3,0}^8) = 0$ and hence $b_{3,0}^8$ and its image under $\unit_{*}^{tmf}$ are nonzero elements in the $E_3$ pages of Adams spectral sequences for $tmf_*$ and $tmf_*(\A \sma D\A)$, respectively.

Since there is a $d_3$ differential $d_3(b_{3,0}^8) = wgr$ in the Adams spectral sequence for $tmf_*$, it will follow that $b_{3,0}^8$ supports a $d_3$-differential in the Adams spectral sequence for $tmf_*(\A \sma D\A)$, provided the image of $wgr$ is nonzero in the $E_3$-page of the Adams spectral sequence for $tmf_*(\A \sma D\A)$. Thus we have to show that there does not exist a differential of the form $d_2(x) = wgr$. 

Using Bruner's program \cite{Bru}, we check that $wgr\in Ext_{A(2)}^{19,95+19}(S^0)$ maps nontrivially to $Ext_{A(2)}^{19,95+19}(\A)$. Thus, if there exists an $x$ such that $d_2(x) = wgr$ in\[Ext_{A(2)}^{s,t}(\A \sma D\A)\To tmf_{t-s}(\A\sma D\A),\]then the image of $x$, call it $x'$, must be nontrivial under the map 
\[j_*: Ext_{A(2)}^{17,96+17}(\A \sma D\A) \to Ext_{A(2)}^{17,96+17}(\A)\]
and we will have $d_2(x') = wgr$ in\[Ext_{A(2)}^{s,t}(A_1)\To tmf_{t-s}(\A).\]There is exactly one generator of $Ext_{A(2)}^{17,96+17}(\A)$, and that generator is $b_{3,0}^4\cdot y$ under the pairing 
\[Ext_{A(2)}^{8,48+8}(S^0)\otimes Ext_{A(2)}^{9,48+9}(\A)\longrightarrow  Ext_{A(2)}^{17,96+17}(\A).\] 
It is clear that $d_2(y) = 0$ as $Ext_{A(2)}^{11,47+11}(\A)=0$ (see Chart~\ref{SSeq}). Thus using the Leibniz rule, we see that 
\[ d_2(b_{3,0}^4y) = e_0r\cdot y.\]
Using \cite{Bru}, we check that $e_0r\cdot y =0$. Therefore, $wgr$ is nonzero in the $E_3$-page of the spectral sequence
\[ Ext_{A(2)}^{s,t}(\A \sma D\A) \To tmf_{t-s}(\A \sma D\A),\]
and therefore
\[d_3(b_{3,0}^8) = wgr\]in this spectral sequence.
\end{proof}
As a consequence of Lemma~\ref{not8or16}, we see that $v_2^8$ and $v_2^{16}$ in $k(2)_*(\A \sma D\A)$ do not lift to $tmf_*(\A \sma D\A)$ and hence cannot lift to $\pi_*(\A \sma D\A)$. Thus we have established:
\begin{thm}\label{nonexistence}The spectra $\A$ do not admit an $8$-periodic or $16$-periodic $v_2$-self-map.\end{thm}

Next we describe an algebraic resolution which will allow us to lift the $d_2$ differential and the $d_3$ differential of Lemma~\ref{not8or16} to the Adams spectral sequence 
\[ E_2^{s,t} = Ext_A^{s,t}(\A \sma D\A) \To \pi_{t-s}(\A \sma D\A).\]

We will briefly recall the resolution described in  \cite[Section~$5$]{BHHM}, and how it is used to lift elements of Ext groups over $A(2)$ to Ext groups over $A$. Consider the $A$-module\[A\modmod A(2):=A\otimes_{A(2)}\Ft\]and denote by $\overline{A\modmod A(2)}$ the kernel of the augmentation map\[A\modmod A(2)\to\Ft.\]
When we consider the triangulated structure of the derived category of $A$-modules, we get maps\[A\modmod A(2)\to\Ft\to\overline{A\modmod A(2)}[1],\]and a resulting diagram
\[\xymatrix{
\Ft\ar[r]&\overline{A\modmod A(2)}[1]\ar[r]&\overline{A\modmod A(2)}^{\otimes 2}[2]\ar[r]&\cdots\\
A\modmod A(2)\ar[u]&A\modmod A(2)\otimes\overline{A\modmod A(2)}[1]\ar[u]&A\modmod A(2)\otimes\overline{A\modmod A(2)}^{\otimes 2}[2]\ar[u]
}\]to which we shall apply the functor $Ext_A^{s,t}(H^*(X)\otimes -,\Ft)$ to get a spectral sequence, which we shall refer to as the algebraic $tmf$ spectral sequence to reflect the fact that $A\modmod A(2)$ is the cohomology of $tmf$. This spectral sequence will be trigraded, with $E_1$ page
\begin{eqnarray*}E_1^{s,t,n}&=&Ext_A^{s,t}(H^*(X)\otimes A\modmod A(2)\otimes\overline{A\modmod A(2)}^{\otimes n}[n],\Ft)\\ &\iso&Ext_{A(2)}^{s-n,t}(H^*(X)\otimes\overline{A\modmod A(2)}^{\otimes n},\Ft)\end{eqnarray*}
which converges to \[ Ext_A^{s,t}(H^*(X), \Ft).\]
For any element in the algebraic $tmf$ spectral sequence in tridegree $(s,t,n)$, we will refer to $s$ as its Adams filtration, $t$ as the internal degree and $n$ as the algebraic $tmf$ filtration. 
The differential $d_r$ has tridegree $(1,0,r)$. It is shown in~\cite{DMext} that\[A\modmod A(2)\iso\bigoplus_{i\geq 0}H^*(\Sigma^{8i}bo_i),\]
where $bo_i$ denotes the $i$-th $bo$-Brown-Gitler spectrum of~\cite{GJM}. As a result the $E_1$ page of the algebraic $tmf$ spectral sequence simplifies to
\[ E_1^{s,t,n}=\bigoplus_{i_1,\ldots,i_n\geq 1}Ext^{s-n,t-8(i_1+\cdots+i_n)}_{A(2)}(X \sma bo_{i_1}\sma\ldots\sma bo_{i_n})\To Ext_A^{s,t}(X).\]
We will attempt to exploit the relative sparseness of the $E_1$ page, especially its vanishing line properties, in the case when $X = \A \sma D\A$. 
\begin{rem}[The cellular structure of $bo$-Brown-Gitler spectra] \label{rem:BG}The spectrum $bo_0$ is the sphere spectrum. The cohomology of the spectrum $bo_1$ as a module over the Steenrod algebra can be described through the following picture, with the generators labelled by cohomological degree:
\[ \begin{tikzpicture}[scale = 1]
\node (a0) at (0,0) [label=above:0] {$\bullet$};
\node (a1) at (4,0) [label=above:4] {$\bullet$};
\node (a2) at (6,0) [label=above:6] {$\bullet$};
\node (a3) at (7,0) [label=above:7] {$\bullet$};
\draw[-] (a2) -- (a3);
\draw[blue, bend right] (a1) to (a2);
\draw[red] (a0) -- (0,-1) -- (4, -1) -- (a1);
\end{tikzpicture} \]
where the black, blue and red lines describe the actions of $Sq^1$, $Sq^2$ and $Sq^4$ respectively.
Note that the 4-skeleton of $bo_1$ is $C\nu$. Indeed, the $bo_i$'s fit together to form the following cofiber sequence 
\[ bo_{i-1} \to bo_{i} \to \Sigma^{4i}B(i)\]
where $B(i)$ is the $i$-th integral Brown-Gitler spectrum as described in~\cite{GJM}. Therefore for every $i\geq 1$, the 7-skeleton of $bo_i$ is $bo_1$ and the 4-skeleton of $bo_i$ is $C\nu.$   
\end{rem}
One can compute $Ext_{A(2)}^{s,t}(A_1 \sma DA_1 \sma bo_i)$ from $Ext_{A(2)}^{s,t}(A_1 \sma DA_1)$ using the Atiyah-Hirzebruch spectral sequence or with Bruner's program~\cite{Bru}. 
\begin{lem}\label{lem:vanishbo} The group\[Ext^{s,t}_{A(2)}(A_1 \sma DA_1 \sma bo_{i_1} \sma \ldots \sma bo_{i_n})\]is zero if $s > \frac{1}{5}((t-s)+6)$.    
\end{lem}   
\begin{proof}We showed in Lemma~\ref{vanishA2} that $Ext_{A(2)}^{s,t}(A_1)$ has a vanishing line $s=\frac{1}{5}(t-s)$ for $t-s \geq 30$ and a vanishing line of $s=\frac{1}{5}(t-s) + 1$ overall. The only generator of $Ext_{A(2)}^{s,t}(A_1)$ with a slope greater than $\frac{1}{5}$ is $h_2$, so if we kill off $h_2$ by considering $Ext_{A(2)}^{s,t}(A_1 \sma C\nu)$ then the vanishing line is precisely $s=\frac{1}{5}(t-s)$.

As we mentioned in Remark~\ref{rem:BG}, the 4-skeleton of any $bo_i$ is $C\nu$ and the next cell is in dimension $6$. So we can build $bo_i$ by attaching finitely many cells to $C\nu$ of dimension $\geq 6$. Hence by using the Atiyah-Hirzebruch spectral sequence and the fact that $\frac{1}{5}(x-6) +1 < \frac{1}{5}x$, one can see that the vanishing line of $A_1 \sma bo_i$ is $s= \frac{1}{5}(t-s)$. One can build $A_1 \sma bo_{i_1} \sma \ldots \sma bo_{i_n}$ from $A_1 \sma bo_{i_1}$, iteratively using cofiber sequences, which depend on the cell structure of $bo_{i_2} \sma \ldots \sma bo_{i_n}$. Since we have already established that $Ext_{A(2)}^{s,t}(A_1 \sma bo_{i_1})$ has vanishing line $s= \frac{1}{5}(t-s)$ and that $bo_{i_2} \sma \ldots \sma bo_{i_n}$ is a connected spectrum, we conclude, using the Atiyah-Hirzebruch spectral sequence, that the vanishing line for $Ext_{A(2)}^{s,t}(A_1 \sma bo_{i_1} \sma \ldots \sma bo_{i_n})$ is 
$s= \frac{1}{5}(t-s)$.

However, $DA_1$ has cells in negative dimension, in fact the bottom cell is in dimension $-6$. Again by using the Atiyah-Hirzebruch spectral sequence, one concludes that the vanishing line for $Ext_{A(2)}^{s,t}(A_1 \sma DA_1 \sma bo_{i_1} \sma \ldots \sma bo_{i_n})$ is 
\[ s = \frac{1}{5}(t-s+6)\]  
for any $i_k \geq 1$, completing the proof.
\end{proof}
\begin{cor} \label{cor:vanishoverA} The group $Ext_{A}^{s,t}(\A \sma D\A)$ is zero if 
\[ s> \frac{1}{5}(t-s) + \frac{11}{5}\]
and for $t-s\leq 23$, it is zero if
\[ s> \frac{1}{5}(t-s) + \frac{6}{5}.\]
\end{cor}
The result is a straightforward consequence of Lemma~\ref{vanishA2}, Lemma~\ref{lem:vanishbo} and the algebraic $tmf$ spectral sequence. 
\begin{lem}\label{v2^8lifts}The element\[b_{3,0}^4\in Ext_{A(2)}^{8,48+8}(\A\sma D\A)\]is in the image of the map\[Ext_A^{8,48+8}(\A\sma D\A)\to Ext_{A(2)}^{8,48+8}(\A\sma D\A).\]\end{lem}
\begin{proof}Clearly $b_{3,0}^4$ is in bidegree $(s,t)=(8,48+8)=(8,56)$ of the $E_1$ page of the algebraic $tmf$ spectral sequence, so we must verify that it is a permanent cycle, which we will do by showing that the $E_1$ page is zero in bidegree $(s,t)=(9,56)$ when $n\geq 1$. Namely, we must show that for every $n\geq 1$, the group
\[\bigoplus_{i_1,\ldots,i_n\geq 1}Ext^{9-n,56-8(i_1+\cdots+i_n)}_{A(2)}(\A\sma D\A\sma bo_{i_1}\sma\ldots\sma bo_{i_n})\]is zero. Using the vanishing line in Lemma~\ref{lem:vanishbo}, the group is zero for all $i_1,\ldots,i_n\geq 1$ such that\[\frac{1}{5}(56-8(i_1+\cdots+i_n)-9+n+6)<9-n\]or
\begin{equation}\label{algtmf:ineq}\frac{1}{5}(53+n-8(i_1+\cdots+i_n))<9-n.\end{equation}
Of course, we have\[\frac{1}{5}(53+n-8(i_1+\cdots+i_n))\leq\frac{1}{5}(53-7n),\]and if $n>4$, we also have\[\frac{1}{5}(53-7n)<9-n.\]

Assume $n=1$, then~\eqref{algtmf:ineq} becomes\[\frac{1}{5}(54-8i_1)<8,\]or\[i_1>1,\]so it suffices to verify that\[Ext_{A(2)}^{8,48}(\A\sma D\A\sma bo_1)=0.\]

Assume $n=2$, then~\eqref{algtmf:ineq} becomes\[\frac{1}{5}(55-8(i_1+i_2))<7,\]or\[i_1+i_2>2,\]so it suffices to verify that\[Ext_{A(2)}^{7,40}(\A\sma D\A\sma bo_1\sma bo_1)=0.\]

Assume $n=3$, then~\eqref{algtmf:ineq} becomes\[\frac{1}{5}(56-8(i_1+i_2+i_3))<6,\]or\[i_1+i_2+i_3>3,\]so it suffices to verify that\[Ext_{A(2)}^{6,32}(\A\sma D\A\sma bo_1\sma bo_1\sma bo_1)=0.\]

Assume $n=4$, then~\eqref{algtmf:ineq} becomes\[\frac{1}{5}(57-8(i_1+i_2+i_3+i_4))<5,\]or\[i_1+i_2+i_3+i_4>4,\]so it suffices to verify that\[Ext_{A(2)}^{5,24}(\A\sma D\A\sma bo_1\sma bo_1\sma bo_1\sma bo_1)=0.\]
For all four models of $\A$, Bruner's program~\cite{Bru} shows that all the groups we expected to be zero are in fact zero.
\end{proof}
\begin{cor}\label{cor:liftA}For all $n\in\mathbb{N}$, the elements $b_{3,0}^{4n} \in Ext_{A(2)}^{8n,48n+8n}(\A \sma D\A)$ lift to $Ext_{A}^{8n,48n+8n}(\A \sma D\A)$ under the map induced by $H_{tmf}$.
\end{cor}
\begin{proof}
Since  $\A\sma D\A$ is a ring spectrum, it follows that the map 
\[ Ext_A^{s,t}(\A \sma D\A) \to Ext_{A(2)}^{s,t}(\A \sma D\A)\]
induced by $H_{tmf}$ is a map of algebras. By Lemma~\ref{v2^8lifts}, $b_{3,0}^4$ lifts and thus $b_{3,0}^{4n}$ lifts for every $n \in \mathbb{N}$. 
\end{proof}
\begin{rem} \label{rem:nonunique2}The lift of $b_{3,0}^{4n} \in Ext_{A(2)}^{8n,48n+8n}(\A \sma D\A)$ to $Ext_{A}^{8n,48n+8n}(\A \sma D\A)$ may not be unique. The conclusions of Lemma~\ref{lem:diffinA} will not depend on the choice of lift.
\end{rem} 
\begin{lem} \label{lem:diffinA} In the Adams spectral sequence 
\[ E_2^{s,t} = Ext_{A}^{s,t}(\A \sma D\A) \To \pi_{t-s}(\A \sma D\A)\] 
there is a $d_2$-differential  
\[ d_2(b_{3,0}^4) = \widetilde{e_0r} = e_0r + R\]
and a $d_3$-differential 
\[ d_3(b_{3,0}^8) = \widetilde{wgr} = wgr +S\]
for some $R$ and $S$ in algebraic $tmf$ filtration greater than zero. 
\end{lem}
\begin{proof}Recall that the element $e_0r\in Ext_{A}^{10,47+10}(S^0)$ (see \cite{Tan}) maps to a nonzero element in $Ext_{A(2)}^{10,47+10}(S^0)$ which is also called $e_0r$ in the literature, and that $d_2(b_{3,0}^4)=e_0r$ in\[Ext_{A(2)}^{s,t}(S^0)\To tmf_{t-s}.\]In Lemma~\ref{not8or16}, we argued that $e_0r$ has a nonzero image under the map 
\[ Ext_{A(2)}^{10,47+10}(S^0) \to Ext_{A(2)}^{10,47+10}(\A \sma D\A). \]
Therefore by inspecting the commutative diagram
\begin{equation}\label{commute1}\xymatrix{ 
Ext_{A}^{10,47+10}(S^0)\ar[d] \ar[r] & Ext_A^{10,47+10}(\A \sma D\A) \ar[d]\\
Ext_{A(2)}^{10,47+10}(S^0) \ar[r]&  Ext_{A(2)}^{10,47+10}(\A \sma D\A),
} \end{equation}
we see that $e_0r \in Ext_{A}^{10,47+10}(S^0)$ has a nonzero image in $Ext_A^{10,47+10}(\A \sma D\A)$. Since $d_2(b_{3,0}^4) = e_0r$ in $Ext_{A(2)}(\A \sma D\A)$, it follows that\[d_2(b_{3,0}^4) = e_0r + R\]in $Ext_{A}(\A \sma D\A)$ for some $R$ in algebraic $tmf$ filtration greater than zero. 

Consequently, $d_2(b_{3,0}^8)=0$ in\[Ext_A^{s,t}(\A\sma D\A)\To\pi_{t-s}(\A\sma D\A),\]and clearly $b_{3,0}^8$ is not hit by a $d_2$ in this spectral sequence, otherwise it would be hit by a differential in\[Ext_{A(2)}^{s,t}(\A\sma D\A)\To tmf_{t-s}(\A\sma D\A).\]However, $b_{3,0}^8$ could support a nonzero $d_3$. The element $wgr\in Ext_{A}^{19,95+19}(S^0)$ maps to a nonzero element of $Ext_{A(2)}^{19,95+19}(S^0)$ we will also call $wgr$. We showed, in Lemma~\ref{not8or16}, that the image of $wgr$ is nonzero in $Ext_{A(2)}^{19,95+19}(\A \sma D\A)$. The diagram 
\begin{equation}\label{commute2}\xymatrix{ 
Ext_{A}^{19,95+19}(S^0)\ar[d] \ar[r] & Ext_A^{19,95+19}(\A \sma D\A) \ar[d]\\
Ext_{A(2)}^{19,95+19}(S^0) \ar[r]&  Ext_{A(2)}^{19,95+19}(\A \sma D\A),
} \end{equation}
makes it clear that the image of $wgr$ is nonzero in $Ext_{A}^{19,95+19}(\A \sma D\A)$.

Note that $wgr \in Ext_{A}^{19,95+19}(\A \sma D\A)$ cannot support a $d_2$-differential as $d_2(wgr)$ would have bidegree $(21, 94+21)$ and 
\[ Ext_{A}^{21, 94+21}(\A \sma D\A) = 0 \]
by Corollary~\ref{cor:vanishoverA}. Moreover, $wgr$ cannot be target of a $d_2$-differential as this will force a $d_2$-differential in $Ext_{A(2)}(\A \sma D\A)$, which is not possible, as we argued in the proof of Lemma~\ref{not8or16}. Thus, $wgr$ is in the $E_3$-page. 

From Lemma~\ref{not8or16}, we know that $d_3(b_{3,0}^8)=wgr$ in the Adams spectral sequence for $tmf_*(\A \sma D\A)$. It follows that
\[ d_3(b_{3,0}^8) = wgr + S,\]
for some $S$ in algebraic $tmf$ filtration greater than zero, in the Adams spectral sequence for  $\pi_{*}(\A \sma D\A)$.
\end{proof}

\section{Another $d_3$ differential} \label{three}
In the Adams spectral sequence\[Ext_{A(2)}^{s,t}(S^0)\To tmf_{t-s},\]there is a well-known $d_3$ differential 
\[ d_3(v_2^{20}h_1) = g^6.\]
The element $g$ is Tangora's name \cite{Tan} for the element detected by $b_{2,1}^2$ in the $E_{\infty}$ page of the May spectral sequence\[\Ft[h_{i,j}:i>0,j\geq 0]\To Ext_A^{s,t}(S^0).\]In the literature, the same name is adopted for its image in $Ext_{A(2)}^{24,120+24}(S^0)$. The goal of this section is to show that this differential induces a $d_3$ differential in\[Ext_{A(2)}^{s,t}(\A\sma D\A)\To tmf_{t-s}(\A \sma D\A)\]and it lifts to a $d_3$ differential under the map of spectral sequences 
\[\xymatrix{
Ext_{A}^{s,t}(\A \sma D\A) \ar[d]  \ar@{=>}[r] & \pi_{t-s}(\A \sma D\A) \ar[d]  \\
Ext_{A(2)}^{s,t}(\A \sma D\A) \ar@{=>}[r] & tmf_{t-s}(\A \sma D\A).
}\]
\begin{lem}\label{d3overA2}In the Adams spectral sequence
\[Ext_{A(2)}^{s,t}(\A\sma D\A)\To tmf_{t-s}(\A\sma D\A),\]
the element $g^6$ is killed by a $d_3$ differential
\[d_3(v_2^{20}h_1)=g^6.\]
\end{lem}
\begin{proof} 
From the calculation in Lemma~\ref{MSS:ExtA2A1}, it is clear that $g^6 = b_{2,1}^{12}$ has a nonzero image in $Ext_{A(2)}^{24,120+24}(\A)$.
Since we have a factorization of maps 
\[ Ext_{A(2)}^{24,120+24}(S^0) \to Ext_{A(2)}^{24,120+24}(\A \sma D\A) \to Ext_{A(2)}^{24,120+24}(\A),\]
$g^6$ must also be nonzero in $Ext_{A(2)}^{24,120+24}(\A\sma D\A)$. Furthermore, because it is hit by a $d_3$ differential in\[Ext_{A(2)}^{s,t}(S^0)\To tmf_{t-s},\]it must also be hit by a $d_3$ differential in\[Ext_{A(2)}^{s,t}(\A\sma D\A)\To tmf_{t-s}(\A\sma D\A).\]

However, this does not preclude the possibility that it might be hit by a $d_2$ differential in this spectral sequence. Indeed, there are elements $\widetilde{x}\in Ext_{A(2)}^{22,121+22}(\A\sma D\A)$ that could support a $d_2$ differential
\[d_2(x)=g^6.\]
In such a case, $x$ would have to map to a nonzero element $x\in Ext_{A(2)}^{22,121+22}(\A)$ and there would exist a differential\[d_2(x)=g^6\]in\[Ext_{A(2)}^{s,t}(\A)\To tmf_{t-s}(\A).\]From the calculations of Lemma~\ref{MSS:ExtA2A1}, there is exactly one possible nonzero $x\in Ext_{A(2)}^{22,121+22}(\A)$. Using Bruner's program~\cite{Bru} (see Equation~\eqref{eqn:lem4.1}) we see that this $x$ is a multiple of $gb_{3,0}^4$ under the pairing 
\begin{eqnarray*}Ext_{A(2)}^{12,68+ 12}(S^{0}) \otimes Ext_{A(2)}^{10,53+10}(\A) &\longrightarrow& Ext_{A(2)}^{22, 121 +22}(\A)\\
   gb_{3,0}^4\otimes\overline{x}&\mapsto&x.
  \end{eqnarray*}
Clearly $d_2(\overline{x}) = 0$ as $Ext_{A(2)}^{9,55+9}(\A) = 0.$ We apply the Leibniz rule to see that 
\[ d_2(x) = ge_0r \cdot \overline{x}.\]
However, $ge_0r= 0$ in $Ext_{A(2)}^{14,67+14}(S^0)$, therefore $d_2(x) = 0$. Consequently, $g^6$ is present and nonzero in the $E_3$ page of the spectral sequence 
\[ Ext_{A(2)}^{s,t}(\A \sma D\A) \To tmf_{t-s}(\A \sma D\A).\] 
Since we have a map of spectral sequences 
\[\xymatrix{
Ext_{A(2)}^{s,t}(S^0) \ar[d]  \ar@{=>}[r] & tmf_{t-s} \ar[d]  \\
Ext_{A(2)}^{s,t}(\A \sma D\A) \ar@{=>}[r] & tmf_{t-s}(\A \sma D\A).
}\]
the result follows. 
\end{proof}
Our next goal is to lift this $d_3$ differential to the Adams spectral sequence\[Ext_A^{s,t}(\A\sma D\A)\To\pi_{t-s}(\A\sma D\A).\]The main tool at our disposal is the algebraic $tmf$ spectral sequence, described in Section~\ref{two}. 
\begin{notn}\label{algtmftables}The elements of $E_{1}^{s,t,n}$, the $E_1$ page of the algebraic $tmf$ spectral sequence for $\A \sma D\A$, which are nonzero permanent cycles, will detect nonzero elements of $Ext_{A}^{s,t}(\A \sma D\A)$. Therefore we place an element $x\in E_1^{s,t,n}$ in bidegree $(t-s-n, s+n)$. Thus the elements that may contribute to the same bidegree of $Ext_{A}^{s,t}(\A \sma D\A)$ are placed together. With this arrangement any differential in the algebraic $tmf$ spectral sequence will look like Adams $d_1$ differential. The generators of
\[E_1^{s,t,n}=\bigoplus_{i_1,\ldots,i_n\geq 1}Ext_{A(2)}^{s-n,t-8(i_1+\cdots+i_n)}(\A\sma D\A\sma bo_{i_1}\sma\ldots\sma bo_{i_n})\]
will be denoted by dots in the following manner (recall that $bo_0=S^0$):
\begin{itemize}
\item elements with $n=0$ are denoted by a $\bullet$,
\item elements with $n=1,i_1=1$ are denoted by a $\circ^1$,
\item elements with $n=1,i_1=2$ are denoted by a $\circ^2$,
\item elements with $n=2,i_1=1,i_2=1$ are denoted by a $\odot$,
\item and $N/A$ stands for `not applicable,' i.e. coordinates of the table which are irrelevant to our arguments.
\end{itemize}
\end{notn}

\begin{lem} \label{lem:lifttoA} The elements $g^6$ and $v_2^{20}h_1$ lift to $Ext_{A}^{s,t}(\A \sma D\A)$ under the map 
\[ \unit_*: Ext_{A}^{s,t}(\A \sma D\A) \longrightarrow Ext_{A(2)}^{s,t}(\A \sma D\A).\]
\end{lem}
\begin{proof}We use the algebraic $tmf$ spectral sequence to show that $g^6$ and $v_2^{20}h_1$ lift to $Ext_{A}(\A \sma D\A)$. A $d_r$ differential in the algebraic $tmf$ spectral sequence will increase the algebraic $tmf$ filtration by $r$. Since $g^6$ and $v_{2}^{20}h_1$ are in algebraic $tmf$ filtration $0$, they cannot be a target of a differential. We will now show that both $g^6$ and $v_2^{20}h_1$ cannot support a nonzero differential. The argument varies for different models of $A_1$. 
\begin{component}[Case 1]When $A_1 = A_1[01]$, Table~\ref{A1-01algtmf120} shows the relevant part of the $E_1$ page of the algebraic $tmf$ spectral sequence. 
\begin{table}[H] 
\caption{$E_1$ page of the algebraic $tmf$ spectral sequence for $Ext_A^{s,t}(A_1\sma DA_1)$, where $\A=\A[01]$} 
\[
\begin{array}{|c|c|c|c|} 
\hline
s\backslash t-s&119&120&121\\ [0.5ex] 
\hline 
25&0&\mbox{N/A}&\mbox{N/A}\\\hline
24&\mbox{N/A}&\lbrace\bullet\bullet\bullet\bullet\rbrace:=X^0_{24}\ni g^6&\mbox{N/A}\\\hline
23&\mbox{N/A}&\mbox{N/A}&\mbox{N/A}\\\hline
22&\mbox{N/A}&0&\mbox{N/A}\\\hline
21&\mbox{N/A}&\mbox{N/A}&\lbrace\bullet\rbrace:=X^0_{21}\ni v_2^{20}h_1\\
&&&\circ^1\circ^1\circ^1\circ^1\circ^1\\
&&&\circ^2\circ^2\circ^2\circ^2\\
&&&\odot\odot\\\hline
\end{array}\] 
\label{A1-01algtmf120}
\end{table}
Elements of $X_{24}^0$ or $X_{21}^0$ in Table~\ref{A1-01algtmf120} clearly do not support a differential, and hence $g^6$ and $v_2^{20}h_1$ lift to $Ext_{A}^{s,t}(\A \sma D\A)$.
\end{component}
\begin{component}[Case 2]The case $\A= \A[10]$ is very similar to the previous one. 
\begin{table}[H] 
\caption{$E_1$ page of the algebraic $tmf$ spectral sequence for $Ext_A^{s,t}(A_1\sma DA_1)$, where $\A=\A[10]$} 
\[\begin{array}{|c|c|c|c|} 
\hline
s\backslash t-s&119&120&121\\ [0.5ex] 
\hline 
25&0&\mbox{N/A}&\mbox{N/A}\\\hline
24&\mbox{N/A}&\lbrace\bullet\bullet\bullet\bullet\rbrace:=X^0_{24}\ni g^6&\mbox{N/A}\\\hline
23&\mbox{N/A}&\mbox{N/A}&\mbox{N/A}\\\hline
22&\mbox{N/A}&0&\mbox{N/A}\\\hline
21&\mbox{N/A}&\mbox{N/A}&\lbrace\bullet\rbrace:=X^0_{21}\ni v_2^{20}h_1\\
&&&\circ^1\circ^1\circ^1\circ^1\circ^1\\
&&&\circ^2\circ^2\circ^2\circ^2\\
&&&\odot\odot\\\hline
\end{array}\] 
\label{A1-10algtmf120}
\end{table}
Elements of $X_{24}^0$ or $X_{21}^0$ in Table~\ref{A1-10algtmf120} clearly do not support a differential, and hence $g^6$ and $v_2^{20}h_1$ lift to $Ext_{A}^{s,t}(\A \sma D\A)$.
\end{component}
\begin{component}[Case 3] The analysis for $\A= \A[00]$ or $\A= \A[11]$ are the same as $A_1[00]$ and $A_1[11]$ are dual to each other. In either case the $E_1$-page of the algebraic $tmf$ spectral sequence around stem $120$ looks like the following:   
\begin{table}[H] 
\caption{$E_1$ page of the algebraic $tmf$ spectral sequence for $Ext_A^{s,t}(A_1\sma DA_1)$, where $\A=\A[00]$ or $\A=\A[11]$} 
\[
\begin{array}{|c|c|c|c|} 
\hline
s\backslash t-s&119&120&121\\ [0.5ex] 
\hline 
25&\bullet&\mbox{N/A}&\mbox{N/A}\\\hline
24&\mbox{N/A}&\lbrace\bullet\bullet\bullet\bullet\bullet\rbrace:=X^0_{24}\ni g^6&\mbox{N/A}\\\hline
23&\mbox{N/A}&\mbox{N/A}&\mbox{N/A}\\\hline
22&\mbox{N/A}&\bullet\bullet\bullet&\mbox{N/A}\\
&&\odot\odot\odot\odot&\\\hline
21&\mbox{N/A}&\mbox{N/A}&\lbrace\bullet\bullet\rbrace:=X^0_{21}\ni v_2^{20}h_1\\
&&&\circ^1\circ^1\circ^1\circ^1\circ^1\\
&&&\circ^2\circ^2\circ^2\circ^2\circ^2\circ^2\\
&&&\odot\odot\odot\odot\odot\odot\\\hline
\end{array}\] 
\label{A1-00algtmf120}
\end{table}
Elements of $X_{24}^0$ in Table~\ref{A1-00algtmf120} clearly do not support a differential, and hence $g^6$ lifts to $Ext_{A}^{s,t}(\A \sma D\A)$. Unfortunately, it is possible that an element of $X_{21}^0$ might support a differential.

However, it is known that $v_2^{20}h_1$ is a multiple of $b_{3,0}^8$ under the pairing 
\begin{eqnarray*}
Ext_{A(2)}^{16,96+16}(S^{0}) \otimes Ext_{A(2)}^{5,25+5}(S^0) &\longrightarrow& Ext_{A(2)}^{21,121+21}(S^0)\\
b_{3,0}^8\otimes v_2^4h_1&\mapsto&v_2^{20}h_1.
\end{eqnarray*}
Therefore the same is true for $v_2^{20}h_1\in Ext_{A(2)}^{21,121+21}(\A \sma D\A)$ as 
\[ \unit_*: Ext_{A(2)}^{s,t}(S^0) \to Ext_{A(2)}^{s,t}(\A \sma D\A)\]
is a map of algebras. By Corollary~\ref{cor:liftA}, we know that $b_{3,0}^8$ lifts to $Ext_{A}^{16,96+16}(\A \sma D\A)$. If we show that $v_2^4h_1$ lifts to $Ext_{A}^{5,25+5}(\A \sma D\A)$ as well, then the result will follow as  
\[ H_{tmf*}: Ext_{A}^{s,t}(\A \sma D\A) \to Ext_{A(2)}^{s,t}(\A \sma D\A)\] 
is a map of algebras. Looking at Table~\ref{A1-00algtmf25} makes it clear that every element of $b_{3,0}^{-8}X^0_{21}$, including $v_2^4h_1$, lifts to $Ext_A^{5,25+5}(A_1\sma DA_1)$, and hence that every element of $X^0_{21}$, including $v_2^{20}h_1$, lifts to $Ext_A^{21,121+21}(A_1\sma DA_1)$.
\begin{table}[H] 
\caption{$E_1$ page of the algebraic $tmf$ spectral sequence for $Ext_A^{s,t}(A_1\sma DA_1)$, where $\A=\A[00]$ or $\A=\A[11]$} 
\[
\begin{array}{|c|c|c|} 
\hline
s\backslash t-s&24&25\\ [0.5ex] 
\hline 
6&\bullet&\bullet\bullet\bullet\\\hline
5&\mbox{N/A}&\lbrace\bullet\bullet\rbrace=b_{3,0}^{-8}X^0_{21}\\
&&\circ^1\\\hline
\end{array}\] 
\label{A1-00algtmf25}
\end{table}
\end{component}
\end{proof}

The lift of $v_2^{20}h_1 \in Ext_{A(2)}^{21,121+21}(\A \sma D\A)$ to $Ext_{A}^{21,121+21}(\A \sma D\A)$ found in Lemma~\ref{lem:lifttoA} is not unique. More precisely, every such lift is
\[v_2^{20}h_1+S \in  Ext_{A}^{21,121+21}(\A \sma D\A)\]
for some element $S$ in the higher algebraic $tmf$ filtration. Notice that the Adams differentials $d_i(S)$ are zero for $i \in \lbrace 2,3\rbrace$ as there are no element of algebraic $tmf$ filtration greater than zero in $Ext_{A}^{10,10+47}(\A \sma D\A)$ and $Ext_{A}^{11,11+47}(\A \sma D\A)$. Therefore the following lemma holds for any choice of lift of $v_2^{20}h_1 \in Ext_{A(2)}^{21,121+21}(\A \sma D\A)$. 
 
\begin{lem}\label{d3}In the Adams spectral sequence
\[Ext_A^{s,t}(\A\sma D\A)\To\pi_{t-s}(\A\sma D\A),\]there exists a $d_3$ differential
\[d_3(v_2^{20}h_1)=g^6.\]
\end{lem}
\begin{proof} Consider the map of Adams spectral sequence
\[ 
\xymatrix{
E_2^{s,t} = Ext_{A}^{s,t}(\A \sma D\A) \ar@{=>}[r] \ar[d] & \pi_{t-s}(\A \sma D\A) \ar[d] \\
E_2^{s,t} = Ext_{A(2)}^{s,t}(\A \sma D\A) \ar@{=>}[r]  & tmf_{t-s}(\A \sma D\A) 
}
\]
induced by $H_{tmf}$. The fact that $g^6$ and $v_2^{20}h_1$ are nonzero in the $E_3$ page of the Adams spectral sequence for $tmf_{*}(\A \sma D\A)$ (see Lemma~\ref{d3overA2}), forces $g^6$ and $v_2^{20}h_1$ have nonzero lift in the $E_3$ page of the Adams spectral sequence for $\pi_*(\A \sma D\A)$. Moreover the map of $E_3$ pages of the spectral sequences commutes with differentials. Thus in the $E_3$ page of the Adams spectral sequence for $\pi_{*}(\A \sma D\A)$
\[ d_3(v_2^{20}h_1) = g^6 + R,\]
where $R$ is an element of algebraic $tmf$ filtration greater than zero. Furthermore, Table~\ref{A1-01algtmf120}, Table~\ref{A1-10algtmf120} and Table~\ref{A1-00algtmf120} make clear that in the bidegree of $g^6$, there are no elements of higher algebraic $tmf$ filtration, and therefore $R=0$.
\end{proof}

\section{$A_1$ admits a 32 periodic $v_2$-self-map}\label{four}

In Section~\ref{two}, we established that the potential candidates for $8$-periodic and $16$-periodic $v_2$-self-maps on $A_1$ support a $d_2$ and a $d_3$ differentials respectively (see Lemma~\ref{lem:diffinA}). So we know by the Leibniz formula that the candidates for $32$-periodic $v_2$-self-map is a nonzero $d_3$-cycle. So the only way these candidates can fail to converge to an element of $\pi_*(A_1\sma DA_1)$ is by supporting a $d_r$ differential for $r \geq 4$ in the Adams spectral sequence 
\[ E_2 = Ext_A(\A \sma D\A) \Rightarrow \pi_*(\A \sma D\A).\]
 So we look for potential targets in $Ext_A^{s,t}(A_1 \sma DA_1)$ when $t-s = 191$ with Adams filtration $s \geq 36$. In order to detect elements with $t-s=191$ we use the \emph{algebraic}-$tmf$ spectral sequence
\[ E_1^{s,t,n} = Ext_{A(2)}^{s-n, t}(\overline{A//A(2)}^{\otimes n}\otimes H^{*}(X),\Z/2). \]

As pointed out in Remark~\ref{rem:nonunique2} the candidates for $32$-periodic $v_2$-self-map may not be unique. To show the existence it is enough to show that one of those candidates is a nonzero permanent cycle in the $E_{\infty}$ page of the Adams spectral sequence.  We conveniently choose $b_{3,0}^{4n} \in Ext_{A}^{8n,48n+8n}(\A \sma D\A)$ to be the lift of $b_{3,0}^{4n} \in Ext_{A(2)}^{8n,48n+8n}(\A \sma D\A)$ whose algebraic $tmf$ filtration is precisely zero.  

Recall that, as an $A(2)$-module
\[ A \modmod A(2) = \bigoplus_{i \in \mathbb{N}} H^{*}(\Sigma^{8i}bo_i)\]
where the $bo_i$ are the $bo$ Brown-Gitler spectra defined by Goerss, Jones and the third author~\cite{GJM}.
Because of this splitting we get  
\begin{eqnarray*}
 E_1^{s,t,n} &=& \bigoplus_{i_1 \ldots, i_n \geq 1}Ext_{A(2)}^{s-n, t-8(i_1 + \ldots + i_n)}(bo_{i_1} \sma \ldots \sma bo_{i_n}\sma \A \sma D\A)
\end{eqnarray*}
for the $E_1$ page of the algebraic $tmf$ spectral sequence.  


%
An easy consequence of the vanishing line established in Lemma~\ref{lem:vanishbo} is the following. 
\begin{lem} \label{lem:vanish}The only potential contributors to $Ext_A^{s,t}(A_1 \sma DA_1)$ for $t-s = 191$ and $s \geq 36$ come from the following summands of the algebraic $tmf$ $E_1$ page:
\begin{eqnarray*}
 			& &	Ext_{A(2)}^{s,t}(A_1 \sma DA_1)\\
			 &\oplus& \bigoplus_{1 \leq i \leq 3} Ext_{A(2)}^{s-1, t-8i}(A_1 \sma DA_1 \sma bo_i)\\
 			&\oplus&  \bigoplus_{1 \leq i \leq 2} Ext_{A(2)}^{s-2, t-8-8i}(A_1 \sma DA_1 \sma bo_1 \sma bo_i) \\ 
			&\oplus&   Ext_{A(2)}^{s-3,t-24}(A_1 \sma DA_1 \sma bo_1 \sma bo_1 \sma bo_1 ).
\end{eqnarray*}
\end{lem}
We know that, in the Adams spectral sequence for $\A \sma D\A$, $b_{3,0}^{16}$ can support $d_r$ differential only if $r \geq 4$. The broad idea is to show that all potential targets for a $d_r$ differential for $r \geq 4$ are either zero or do not lift to $E_4$ page. While the result holds for all models of $\A$, the computations will be slightly different for different models, and so we will treat these models separately. Since $\A[00]$ and $\A[11]$ are Spanier-Whitehead dual to each other, we can treat the cases of $\A[00]$ and $\A[11]$ as one case. We will then have to treat the cases of the selfdual spectra $\A[01]$ and $\A[10]$ separately. The completeness of the tables in this section will be justified by the more detailed tables in Appendix~\ref{appendixC}.
\subsection{The case $\A=\A[00]$ or $\A=\A[11]$}
We begin by laying out, in Table~\ref{A1-00algtmf191}, the elements of the $E_1$ page of  algebraic $tmf$ spectral sequence, in Notation~\ref{algtmftables}. The table makes it clear that all elements with $t-s = 191$, with the possible exception of those in $X^0_{36}$, are permanent cycles in the algebraic $tmf$ spectral sequence.
\begin{table}[H] 
\caption{$E_1$ page of the algebraic $tmf$ spectral sequence for $Ext_A^{s,t}(A_1\sma DA_1)$, where $\A=\A[00]$ or $\A=\A[11]$, stem $189$-$191$} 
\centering
\begin{tabular}{|c|c|c|c|} 
\hline 
$s\backslash t-s$ & 189& 190 & 191 \\ [0.5ex] 
\hline 
40 &0& 0& 0 \\ \hline 
39& 0&$\lbrace\bullet\bullet\rbrace:=Y_{39}^0$ & $\lbrace\bullet\bullet\bullet\rbrace:=X_{39}^0$\\\hline
38& N/A&$\lbrace\bullet\bullet\bullet\bullet\bullet\rbrace:=Y_{38}^0$ & $\lbrace\bullet\bullet\bullet\rbrace:=X_{38}^0$\\ \hline
37& N/A&$\bullet\bullet\bullet\bullet\bullet$ & $\lbrace\bullet\bullet\bullet\bullet\bullet\rbrace:=X_{37}^0$\\ 
 & &$\circ^1\circ^1\circ^1\circ^1\circ^1\circ^1$ &$\lbrace\circ^1\circ^1\circ^1\circ^1\circ^1\circ^1\circ^1\circ^1\rbrace:=X_{37}^1$\\ \hline
36&N/A& N/A& $\lbrace\bullet\bullet\bullet\rbrace:=X_{36}^0$\\ 
 & & &$\lbrace\circ^1\circ^1\rbrace:=X_{36}^1$\\
 & & &$\lbrace\odot\odot\odot\odot\odot\odot\rbrace:=X_{36}^{1,1}$\\ \hline
\end{tabular} 
\label{A1-00algtmf191}
\end{table}
Our goal is to show that every linear combination of elements in $X^{i_1,\ldots,i_n}_s$ were either absent or zero in the $E_4$ page of the Adams spectral sequence. Using Bruner's program (for details see Tables~\ref{ADA1-00},~\ref{ADA1-00bo1},~\ref{ADA1-00bo2}, and~\ref{ADA1-00bo1bo1} of Appendix~\ref{appendixC}), we observe that a lot of these elements are multiples  of $g^6$ in the $E_1$ page of the algebraic $tmf$ spectral sequence, which we record in Table~\ref{A1-00algtmf71}.  
\begin{table}[H] 
\caption{$E_1$ page of the algebraic $tmf$ spectral sequence for $Ext_A^{s,t}(A_1\sma DA_1)$, where $\A=\A[00]$ or $\A=\A[11]$, stem $70$-$71$} 
\centering
\begin{tabular}{| c | c | c |} 
\hline
$s\backslash t-s$ & 70 & 71 \\ [0.5ex] 
\hline 
15&$\lbrace\bullet\bullet\rbrace=g^{-6}Y_{39}^0$&$\lbrace\bullet\bullet\bullet\rbrace=g^{-6}X_{39}^0$\\\hline
14&$\lbrace\bullet\bullet\bullet\bullet\bullet\rbrace=g^{-6}Y_{38}^0$&$\lbrace\bullet\bullet\bullet\bullet\rbrace=g^{-6}X_{38}^0$\\ \hline
13&$\bullet\bullet\bullet\bullet\bullet$&$\lbrace\bullet\bullet\bullet\bullet\bullet\bullet\bullet\rbrace=g^{-6}X_{37}^0$\\ 
 &$\circ^1\circ^1\circ^1\circ^1\circ^1\circ^1$&$\lbrace\circ^1\circ^1\circ^1\circ^1\circ^1\circ^1\circ^1\circ^1\rbrace=g^{-6}X_{37}^1$\\ \hline
12&N/A&$\lbrace\circ^1\circ^1\circ^1\circ^1\circ^1\circ^1\rbrace=g^{-6}X_{36}^1$\\
 & &$\lbrace\odot\odot\odot\odot\odot\odot\rbrace=g^{-6}X_{36}^{1,1}$\\ \hline
\end{tabular} 
\label{A1-00algtmf71}
\end{table}
 Tables~\ref{ADA1-00},~\ref{ADA1-00bo1},~\ref{ADA1-00bo2}, and~\ref{ADA1-00bo1bo1} make clear that
\begin{itemize}
\item multiplication by $g^6$ surjects onto $X_{39}^0\oplus X_{38}^0\oplus X_{37}^0\oplus X_{37}^1\oplus X_{36}^1\oplus X_{36}^{1,1}$, and
 \item  Elements in $g^{-1}(X_{39}^0\oplus X_{38}^0\oplus X_{37}^0\oplus X_{37}^1\oplus X_{36}^1\oplus X_{36}^{1,1})$ have nonzero images under multiplication by $v_2^{20}h_1$ if and only if multiplication by $g^6$ is nonzero.
\end{itemize}
\begin{lem}\label{4.2}
Every element of\[X_{39}^0\oplus X_{38}^0\oplus X_{37}^0\oplus X_{37}^1\oplus X_{36}^1\oplus X_{36}^{1,1}\]is present in the Adams $E_2$ page, but is either zero or absent in the Adams $E_4$ page.
\end{lem}
\begin{proof} Notice that for any $x =g^6 \cdot y\in X_{39}^0\oplus X_{38}^0\oplus X_{37}^0\oplus X_{37}^1\oplus X_{36}^1\oplus X_{36}^{1,1}$, both $x$ and $y$ is a nonzero permanent cycle of the algebraic $tmf$ spectral sequence. Indeed, the target of any differential supported by $y$, must have algebraic $tmf$ filtration greater than $y$ and from Table~\ref{A1-00algtmf71} it is clear no such element is present in appropriate bidegree. Hence $y$ is present in the Adams $E_2$ page. Same argument holds for $x$.
\begin{component}[Case 1] When $x =g^6 \cdot y \in X_{39}^0\oplus X_{38}^0$, clearly  $y$ is then a permanent cycle in the Adams spectral sequence. Using Leibniz rule, we see that   
\[ d_2(x) = d_2(g^6 \cdot y) = 0\]
and 
\[d_3(v_2^{20}h_1 \cdot y)=v_2^{20}h_1 \cdot d_3(y)+d_3(v_2^{20}h_1)\cdot y=g^6 \cdot y = x.\] 
Therefore, if $x=g^6 \cdot y$ is nonzero in $E_3$ page, then $x$ is zero in $E_4$ page.  
\end{component} 
\begin{component}[Case 2] When $ x = g^6\cdot y\in  X_{37}^1\oplus X_{36}^1\oplus X_{36}^{1,1}$, then $d_r(y)$ for $r \geq 2$, if nonzero, must have algebraic $tmf$ filtration greater than zero, as  
\[ 
\xymatrix{
Ext_{A}(\A \sma D\A)\ar[d] \ar@{=>}[r] & \pi_*(\A \sma D\A) \ar[d] \\
Ext_{A(2)}(\A \sma D\A) \ar@{=>}[r] & tmf_*(\A \sma D\A)
}
\]
is a map of spectral sequence. Since there are no elements of algebraic $tmf$ filtration greater than zero in bigree $(s, 71+s)$ for $s\geq 14$, it follows that $d_r(y) =0$ for $r \geq 2$ and $y$ a permanent cycle in the Adams spectral sequence. If $y$ is a target of a differential in algebraic $tmf$ spectral sequence or a Adams $d_2$ differential, then $x=0$ in $E_3$ page. Consequently, $g^6x= 0$ in the $E_3$ page as well. If $y$ is not a target of such differentials, then we have 
\[d_3(v_2^{20}h_1\cdot y)=v_2^{20}h_1\cdot d_3(y)+d_3(v_2^{20}h_1) \cdot y=g^6\cdot y =x.\]
In either case, $g^6\cdot y = x =0$ in $E_4$ page. 
\end{component} 
\begin{component}[Case 3] When $x = g^6\cdot y \in X_{37}^0$ and $y$ is a permanent cycle, then we can argue  $g^6\cdot y = x =0$ in the $E_4$ page as we did in the previous cases. If 
\[ d_2(y) =y'\]
then $y'$ must belong to  $g^{-1}Y_{39}^{0}$. Since multiplication by $g^6$ is a bijection between $g^{-1}Y_{39}^{0}$ and $Y_{39}^{0}$, we get 
\[d_2(x) = d_2(g^6\cdot y)=g^6\cdot d_2(y)+d_2(g^6) \cdot y=g^6 \cdot y' \neq 0.\]
Therefore, $x$ is absent in the $E_4$ page.  
\end{component}
\end{proof}





Thus we are left with the case when $x \in X^0_{36}$.
\begin{lem}\label{4.3}
Every element of $X^0_{36}$ is either zero or absent in the Adams $E_4$ page.
\end{lem}
\begin{proof} $X^0_{36}$ is spanned by three generators $\lbrace s_1, t_1,t_2\rbrace$. Using Bruner's program (see ), we explore the following relations:
\[\begin{array}{rcl}
s_1 &=&b_{3,0}^4 \cdot x_1\\
t_1&=&b_{3,0}^4 \cdot y_1=b_{3,0}^8 \cdot z_1\\
t_2 &=&b_{3,0}^4\cdot y_2 =b_{3,0}^8 \cdot z_2\\
Y^0_{38} \ni e_0r \cdot x_1&\neq & 0\\
e_0r \cdot y_1&=& 0\\
e_0r \cdot y_2&=& 0\\
Y^0_{39}\ni wgr \cdot z_1&\neq & 0\\
Y^0_{39}\ni wgr \cdot z_2&\neq& 0
\end{array}\]
and $wgr \cdot z_1$ and $wgr \cdot z_2$ are linearly independent. In Bruner's notation, $s_1 = 36_{64}$, $t_1 = 36_{65}$, $t_2 =36_{66}$, $x_1 = 28_{32}$, $e_0r \cdot x_1 =38_{25}$, $y_1 =28_{33}$, $y_2 =28_{34}$, $z_1 = 20_{1}$, $wgr \cdot z_1 = 39_1$, $z_2 = 20_{2}$ and $wgr \cdot z_2 = 39_2$ (see Table~\ref{Lemma5.3:ADA1-00}) . 
\begin{table}[H] 
\caption{$E_1$ page of the algebraic $tmf$ spectral sequence for $Ext_A^{s,t}(A_1\sma DA_1)$, where $\A=\A[00]$ or $\A=\A[11]$} 
\[
\begin{array}{| c | c | c |} 
\hline
s\backslash t-s&94&95\\\hline
23&0&0\\\hline
22&0&0\\\hline
21&0&0\\\hline
20&\mbox{N/A}&\lbrace\bullet=z_1,\bullet = z_2\rbrace:=Z_{20}\\\hline\hline
s\backslash t-s&142&143\\\hline
30&0&0\\\hline
29&\bullet\bullet\bullet\bullet\bullet&\bullet\bullet\bullet\bullet\bullet\\\hline
28&\mbox{N/A}&\lbrace\bullet =x_1,\bullet=y_1,\bullet=y_2\rbrace:=Z_{28}\\
&&\circ^1\circ^1\\
\hline 
\end{array} \]
\label{A1-00algtmf95143}
\end{table}
From the Table~\ref{A1-00algtmf95143}, it is clear that any element in $Z_{20}$ and $Z_{28}$ are permanent cycles.
\begin{component}[Case 1] If $x = \epsilon_1 s_1 + \delta_1t_1 + \delta_2t_2 \neq 0$ in the Adams $E_2$ page with $\epsilon_1 \neq 0$, then 
\[ d_2(x) = \epsilon_1 (e_0r \cdot x_1) \neq 0.\]
Thus $x$ is not present in $E_4$ page.
\end{component}
\begin{component}[Case 2] If $x = \delta_1t_1 + \delta_2t_2 \neq 0$, then 
\[ d_2(x) = 0.\]
 If $x \neq 0$ in the Adams $E_3$ page then 
\[d_3(x) = \delta_1 d_3(b_{3,0}^4 \cdot z_1)+ \delta_2 d_3(b_{3,0}^4 \cdot z_2) =wgr \cdot (\delta_1z_1 + \delta_2 z_2) \neq 0  \]
Thus $x$ is not present in $E_4$ page.
\end{component}

\end{proof}

This proves Theorem~\ref{mainthm} in the cases $\A=\A[00]$ or $\A=\A[11]$.

\subsection{The case $\A=\A[01]$ or $\A=\A[01]$ } Even though, in principle, we should treat $\A[01]$ and $\A[10]$ as two different cases, but it turns out that Tables~\ref{A1-01algtmf191}, and \ref{A1-01algtmf71} are identical in both the case and the arguments remain exactly the same for both of them. For $\A[01]$, refer to Tables~\ref{ADA1-01},~\ref{ADA1-01bo1},~\ref{ADA1-01bo2}, and~\ref{ADA1-01bo1bo1} of Appendix~\ref{appendixC}, and for $\A[10]$, refer to Tables~\ref{ADA1-10},~\ref{ADA1-10bo1},~\ref{ADA1-10bo2}, and~\ref{ADA1-10bo1bo1} of Appendix~\ref{appendixC}, to observe that most of the elements in Table~\ref{A1-01algtmf191} are multiples by $g^6$ of elements in Table~\ref{A1-01algtmf71}. 

\begin{table}[H] 
\caption{$E_1$ page of the algebraic $tmf$ spectral sequence for $Ext_A^{s,t}(A_1\sma DA_1)$, where $\A=\A[01]$} 
\centering
\begin{tabular}{| c | c | c |} 
\hline
$s\backslash t-s$ & 190 & 191 \\ [0.5ex] 
\hline 
39&0& $\lbrace\bullet\rbrace:=X_{39}^0$\\\hline
38&$\lbrace\bullet\bullet\bullet\bullet\rbrace:=Y_{38}^0$ & $\lbrace\bullet\rbrace:=X_{38}^0$\\ \hline
37&$\bullet\bullet\bullet\bullet$ & $\lbrace\bullet\bullet\bullet\bullet\bullet\rbrace:=X_{37}^0$\\ 
 &$\circ^1\circ^1\circ^1\circ^1\circ^1\circ^1$ &$\lbrace\circ^1\circ^1\circ^1\circ^1\circ^1\circ^1\circ^1\circ^1\rbrace:=X_{37}^1$\\ \hline
36&N/A&$\lbrace\odot\odot\rbrace:=X_{36}^{1,1}$\\ \hline
\end{tabular} 
\label{A1-01algtmf191}
\end{table}

\begin{table}[H] 
\caption{$E_1$ page of the algebraic $tmf$ spectral sequence for $Ext_A^{s,t}(A_1\sma DA_1)$, where $\A=\A[01]$} 
\centering
\begin{tabular}{| c | c | c |} 
\hline
$s\backslash t-s$ & 70 & 71 \\ [0.5ex] 
\hline 
15&0& $\lbrace\bullet\rbrace=g^{-6}X_{39}^0$\\\hline
14&$\lbrace\bullet\bullet\bullet\bullet\rbrace=g^{-6}Y^0_{38}$ & $\lbrace\bullet\bullet\rbrace=g^{-6}X_{38}^0$\\\hline
13&$\bullet\bullet\bullet\bullet\bullet$ & $\lbrace\bullet\bullet\bullet\bullet\bullet\bullet\bullet\rbrace=g^{-6}X^0_{37}$\\ 
 &$\circ^1\circ^1\circ^1\circ^1\circ^1\circ^1$ &$\lbrace\circ^1\circ^1\circ^1\circ^1\circ^1\circ^1\circ^1\circ^1\rbrace=g^{-6}X^1_{37}$\\\hline
12&N/A&$\circ^1\circ^1\circ^1\circ^1\circ^1\circ^1$\\
 & &$\lbrace\odot\odot\rbrace=g^{-6}X^{1,1}_{36}$\\ \hline
\end{tabular} 
\label{A1-01algtmf71}
\end{table}
Tables~\ref{ADA1-01},~\ref{ADA1-01bo1},~\ref{ADA1-01bo2}, and~\ref{ADA1-01bo1bo1} and Tables~\ref{ADA1-10},~\ref{ADA1-10bo1},~\ref{ADA1-10bo2}, and~\ref{ADA1-10bo1bo1} make clear that
\begin{itemize}\item multiplication by $g^6$ is surjective onto $X_{39}^0\oplus X_{38}^0\oplus X_{37}^0\oplus X_{37}^1 \oplus X_{36}^{1,1}$, and
 \item elements in $g^{-6}(X_{39}^0\oplus X_{38}^0\oplus X_{37}^0\oplus X_{37}^1 \oplus X_{36}^{1,1})$ have nonzero images under multiplication by $v_2^{20}h_1$ if and only if multiplication by $g^6$ is nonzero.\end{itemize}
\begin{lem}\label{4.3}
 All elements of\[X_{39}^0\oplus X_{38}^0\oplus X_{37}^0\oplus X_{37}^1\oplus X_{36}^{1,1}\]are present in the Adams $E_2$ page, but are zero in the Adams $E_4$ page.
\end{lem}
\begin{proof}
Notice that for any $x =g^6 \cdot y \in X_{39}^0\oplus X_{38}^0\oplus X_{37}^0\oplus X_{37}^1 \oplus X_{36}^{1,1}$, both $x$ and $y$ is a nonzero permanent cycle of the algebraic $tmf$ spectral sequence. Indeed, the target of any differential supported by $y$, must have algebraic $tmf$ filtration greater than $y$ and from Table~\ref{A1-01algtmf71} it is clear no such elements are present in appropriate bidegrees. Hence $y$ is present in the Adams $E_2$ page. Same argument holds for $x$.

Any $y \in g^{-6}(X_{39}^0\oplus X_{38}^0\oplus X_{37}^0\oplus X_{37}^1)$ is a permanent cycle in Adams spectral sequence, as it is clear from Table~\ref{A1-01algtmf71} that $Ext^{s, 70+s}_{A}(\A \sma D\A) = 0$ for $s \geq 15$. If  $y \in g^{-6}X_{36}^{1,1}$, then $y$ has algebraic $tmf$ filtration greater than zero, therefore $d_r(y)$ must have algebraic $tmf$ filtration greater than zero. From Table~\ref{A1-01algtmf71}, we observe that $Ext^{s, 70+s}_{A}(\A \sma D\A)$ when $s \geq 14$, does not contain any element of algebraic $tmf$ filtration greater than zero. Therefore, any $y \in g^{-6}X_{36}^{1,1}$ is a permanent cycle as well. 

Since $d_2(g^6) =0$, for any $x =g^6 \cdot y \in X_{39}^0\oplus X_{38}^0\oplus X_{37}^0\oplus X_{37}^1\oplus X_{36}^{1,1}$
\[ d_2(x)  = d_2(g^6 \cdot y) = d_2(g^6) \cdot y + g^6 \cdot d_2(y) = 0. \]
Hence $x$ is present in $E_3$ page. 

 If $x = g^6 \cdot y$ is nonzero in $E_3$ page, Bruner's program shows that $v_2^{20}h_1 \cdot y$ is nonzero as well. Thus, using Leibniz rule
\[ d_3(v_2^{20}h_1 \cdot y) = d_3(v_2^{20}h_1) \cdot y + v_2^{20}h_1 \cdot d_3( y) = g^6 \cdot y = x.\]
Thus, $x$ is zero in $E_4$ page.
\end{proof}

\newpage
\appendix
\section{General remarks on the use of Bruner's program}\label{appendix}
\setcounter{table}{0}
\setcounter{figure}{0} 
Since many of our proofs relied on the output of Bruner's program, we append some facts about the program to justify our claims.

The program takes as input a graded module $M$ over $A$ (or $A(2)$) that is a finite dimensional $\Ft$-vector space and computes $Ext_A^{s,t}(M, \Ft)$ (or $Ext_{A(2)}^{s,t}(M, \Ft)$) for $t$ in a user-defined range, and $0\leq s\leq \texttt{MAXFILT}$, where one has $\texttt{MAXFILT}=40$ by default. The structure of $M$ as an $A$-module is encoded in a text file named \texttt{M}, placed in the directory \texttt{A/samples} in the way we will now describe.

The first line of the text file \texttt{M} consists of a positive integer $n$, the dimension of $M$ as an $\Ft$-vector space, whose basis elements we will call $g_0,\ldots ,g_{n-1}$. The second line consists of an ordered list of integers $d_0,\ldots,d_{n-1}$, which are the respective degrees of the $g_i$. Every subsequent line in the text file describes a nontrivial action of some $Sq^k$ on some generator $g_i$. For instance, if we have\[Sq^k(g_i)=g_{j1}+\cdots +g_{jl},\]we would encode this fact by writing the line\[\texttt{i k l j1 \ldots jl}\]followed by a new line. Every action not encoded by such a line is assumed to be trivial.
To ensure that such a text file in fact represents an honest $A$-module, we must run the \texttt{newconsistency} script, which will alert us if:
\begin{itemize}
\item the text file contains a line\[\texttt{i k l j1 \ldots jl}\]and it turns out that one of the $d_j$'s is not equal to $d_i+k$, or
\item the module taken as a whole fails to satisfy a particular Adem relation.
\end{itemize}

\begin{ex}\label{A1coh}
Consider the $A$-module given by Figure~\ref{A1}, where generators are depicted by dots and actions of $Sq^1$, $Sq^2$, and $Sq^4$ are depicted by black, blue and red lines respectively:
\begin{figure}[!h]
\centering
\begin{tikzpicture}[scale = .7]
\node (g0) at (2,0) [label=right:$g_0$] {$\bullet$};
\node (g1) at (2,1) [label=right:$g_1$] {$\bullet$};
\node (g2) at (4,2) [label=right:$g_2$] {$\bullet$};
\node (g3) at (4,3) [label=right:$g_3$] {$\bullet$};
\node (g7) at (4,6) [label=right:$g_7$] {$\bullet$};
\node (g6) at (4,5) [label=right:$g_6$] {$\bullet$};
\node (g5) at (2,4) [label=left:$g_5$] {$\bullet$};
\node (g4) at (2,3) [label=left:$g_4$] {$\bullet$};
\draw[-] (g0) -- (g1);
\draw[-] (g2) -- (g3);
\draw[-] (g4) -- (g5);
\draw[-] (g6) -- (g7);
\draw[blue, bend left] (g1) to (g4);
\draw[blue, bend right] (g3) to (g6);
\draw[blue, out= 20, in = -160] (g0) to (g2);
\draw[blue, out= 160, in = -20] (g2) to (g5);
\draw[blue, out= 20, in = -160] (g5) to (g7);
\draw[red] (g1) -- (0.8,1) -- (0.8,5) -- (2.9,5) (3.1,5) -- (g6);
\end{tikzpicture}
\caption{$H^*A_1[00]$ as an $A$-module}
\label{A1}
\end{figure}
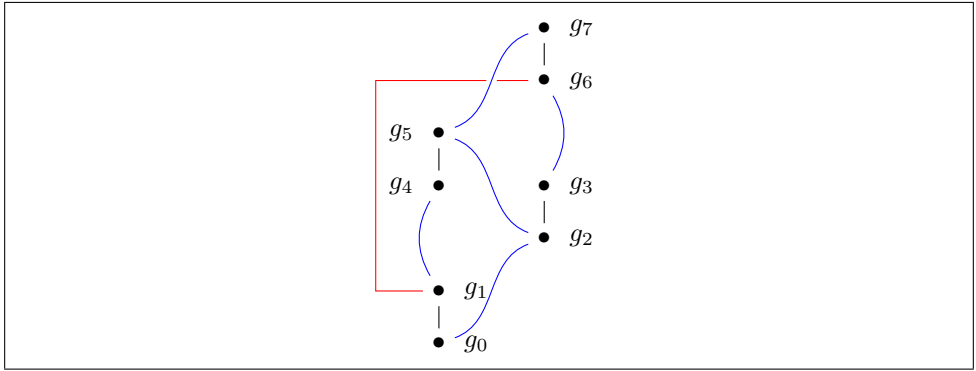
\end{ex}
Based on this picture, we get the text file in figure~\ref{A1-00-def}, which we call \texttt{A1-00\_def}:
\begin{figure}[h] 
\begin{small}
\begin{verbatim}
8

0 1 2 3 3 4 5 6

0 1 1 1
0 2 1 2
1 2 1 4
1 4 1 6
2 1 1 3
2 2 1 5
3 2 1 6
4 1 1 5
5 2 1 7
6 1 1 7
\end{verbatim}
\end{small}
\caption{The text file \texttt{A/samples/A1-00\_def}}
\label{A1-00-def}
\end{figure}
We go to the directory \texttt{A2} and run
\begin{center}
\[ \begin{array}{l}
\texttt{./newmodule A1-00 ../A/samples/A1-00\_def}\\
\texttt{cd A1-00.}
\end{array} \]
\end{center}
Now we are ready to compute. Running the script\begin{center}\texttt{./dims 0 250\&}\end{center}will compute $Ext_{A(2)}^{s,t}(A_1[00])$ for $0\leq s\leq\texttt{MAXFILT}=40$ and $0\leq t \leq 250$. The \texttt{\&} is not strictly necessary, but may be a good idea if running a computation expected to take a long time and if one would like to do other things in the meantime. Then, to see the Ext group, one runs
\begin{center}
\[\begin{array}{l}
\texttt{./report summary}\\\texttt{./vsumm A1-00 > A1-00.tex}\\\texttt{pdflatex A1-00.tex}
\end{array} \]
\end{center}
to produce a pdf document \texttt{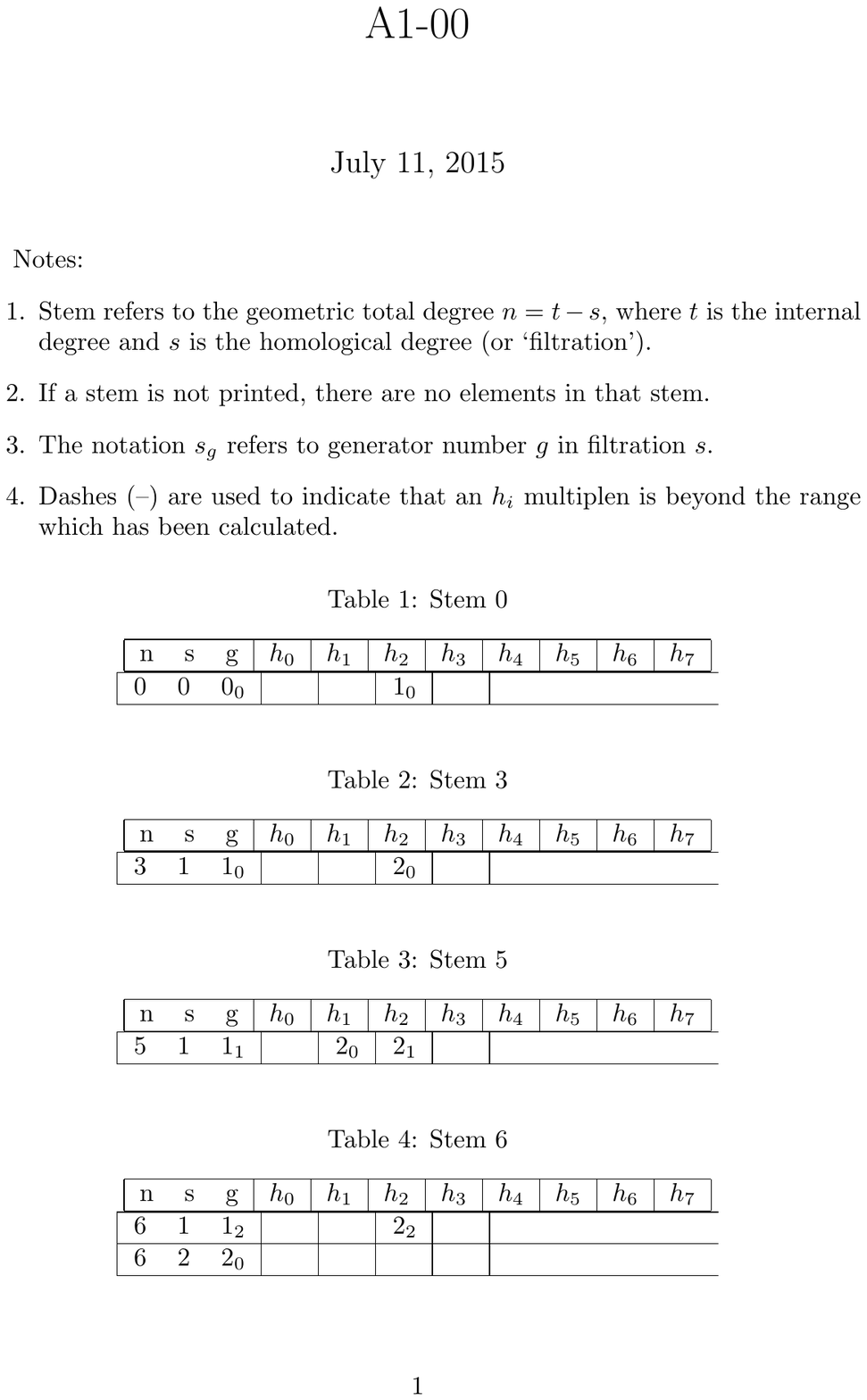} resembling Figure~\ref{A1-00-pdf}.
\begin{figure}
 \centering
 \includegraphics[page=1,width=\textwidth]{A1-00.pdf}
 \caption{First page of \texttt{A1-00.pdf}}
\label{A1-00-pdf}
\end{figure}
As the figure makes apparent, the generators of the Ext group (as an $\Ft$ vector space) are stored in the computer with names such as $s_g$, where $s$ is the Adams filtration of the generator, and $g$ is some way of ordering all generators of filtration $s$. It should be emphasized that $g$ is not the stem of the generator. In figure~\ref{A1-00-pdf} for instance, the generator $1_2$ is the second generator of filtration 1, but it is in stem 6. The figure also tells us the action of the Hopf elements $h_0$ through $h_3$, so that in our example, $h_2$ multiplied by the generator $1_2$ equals the generator $2_2$.

By running
\begin{center}
\texttt{./display 0 A1-00\_ }
\end{center}
to produce single-page pdf documents \texttt{A1-00\_1.pdf}, \texttt{A1-00\_2.pdf}..., it is also possible to see the Ext group in the visually more appealing form of a chart, as shown in figure~\ref{A1-00-1-pdf}. What is gained in esthetics is however lost in completeness, as these charts only display the action of $h_0$ (via a vertical solid line), $h_1$ (via a solid line of slope 1), and $h_2$ (via a dotted line of slope $\frac{1}{3}$).

\begin{figure}
 \centering
 \includegraphics[page=1,width=\textwidth]{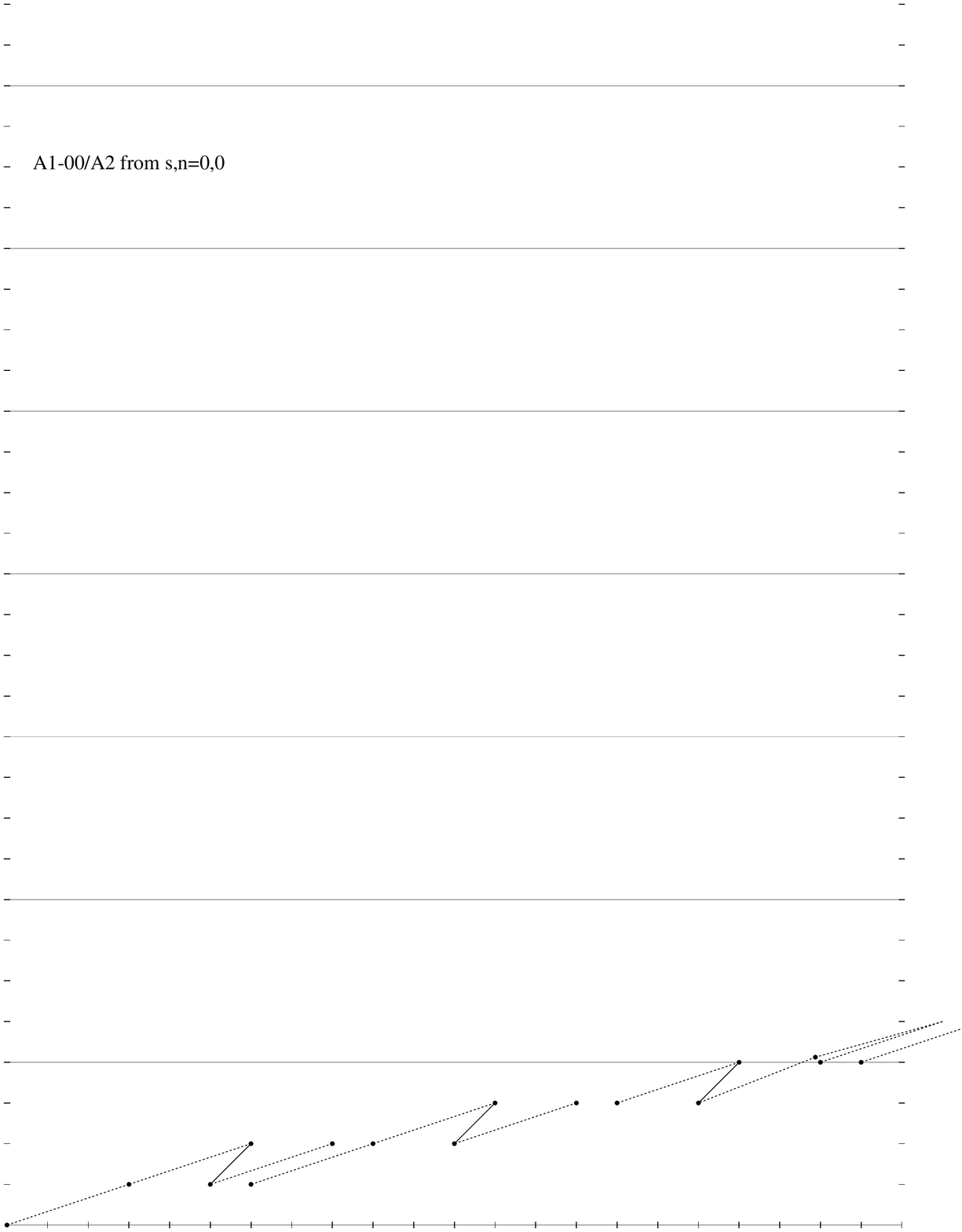}
 \caption{The file \texttt{A1-00\_1.pdf}}
\label{A1-00-1-pdf}
\end{figure}

The program is also capable of computing dual modules via the \texttt{dualizeDef} script, and tensor products via the \texttt{tensorDef} script. Both executables are conveniently located in the \texttt{A/samples} directory where we put our module definition text files. Thus, running
\begin{center}
\[ \begin{array}{l}
 \texttt{./dualizeDef A1-00\_def DA1-00\_def}\\
 \texttt{./tensorDef A1-00\_def DA1-00\_def ADA1-00\_def},
\end{array} \]
\end{center}
produces the text file \texttt{ADA1-00\_def}, with which we proceed in the same way as earlier with \texttt{A1-00\_def}.

While \texttt{ADA1-00.pdf} only shows the action of the Hopf elements $h_0$ through $h_3$, the scripts \texttt{cocycle} and \texttt{dolifts} enable the user to input a specific generator and find the action of much of $Ext_{A(2)}^{s,t}(S^0)$ on that specific generator. Let us do this with the generator $0_6\in Ext_{A(2)}^{0,0}(A_1[00]\sma DA_1[00])$ by going to directory \texttt{A2} and running
\begin{center}
\texttt{./cocycle ADA1-00 0 6}
\end{center}
which will create a subdirectory \texttt{A2/ADA1-00/0\_6}.
To find the action of all elements of $Ext_{A(2)}^{s,t}(S^0)$ with $0\leq s\leq 20$ on $0_6$, we go back to directory \texttt{A2/ADA1-00} and run
\begin{center}
\texttt{./dolifts 0 20 maps}
\end{center}
Now \texttt{ADA1-00/0\_6} will contain several text files, among them \texttt{brackets.sym} (which contains information about Massey products) and \texttt{Map.aug} (which contains information about the action of $Ext_{A(2)}^{s,t}(S^0)$ on $0_6$).

The generators of $Ext_{A(2)}^{s,t}(S^0)$ are stored in the computer in the format $s_g$. In figure~\ref{sgtable}, we include a list of important elements of $Ext_{A(2)}^{s,t}(S^0)$ and their $s_g$ representation.
\begin{figure}
\[\begin{array}{rcl}
g&=&4_8\in Ext_{A(2)}^{4,20+4}(S^0)\\
b_{3,0}^4&=&8_{19}\in Ext_{A(2)}^{8,48+8}(S^0)\\
e_0r&=&10_{18}\in Ext_{A(2)}^{10,47+10}(S^0)\\
b_{3,0}^8&=&16_{54}\in Ext_{A(2)}^{16,96+16}(S^0)\\
wgr&=&19_{56}\in Ext_{A(2)}^{19,95+19}(S^0)\\
v_2^{20}h_1&=&21_{85}\in Ext_{A(2)}^{21,121+21}(S^0)\\
g^6&=&24_{90}\in Ext_{A(2)}^{24,120+24}(S^0)\\
\end{array}\]
\caption{$s_g$ representation of important elements of $Ext_{A(2)}^{s,t}(S^0)$}\label{sgtable}
\end{figure}

We'd like to know what $s_g(0_6)\in Ext_{A(2)}(A_1[00]\sma DA_1[00])$ is in the notation of \texttt{ADA1-00.pdf}. Of course, $s_g(0_6)$ is in filtration $s$, so we only need to specify which of the generators in filtration $s$ make up $s_g(0_6)$. If, for instance, we have\[s_g(0_6)=s_{g1}+\ldots +s_{gn},\]then \texttt{ADA1-00/0\_6/Map.aug} will contain the lines
\begin{center}
\texttt{s g1 g}\\\texttt{s g2 g}\\\texttt{...}\\\texttt{s gn g}.
\end{center}

Now, in the Adams spectral sequence\[Ext_{A(2)}^{s,t}(S^0)\Rightarrow tmf_{t-s},\]we have\[d_2(b_{3,0}^4)=e_0r=10_{18}\in Ext_{A(2)}^{10,47+10}(S^0),d_3(b_{3,0}^8)=19_{56}\in Ext_{A(2)}^{19,95+19}(S^0).\]
It follows that if $10_{18}(0_6)=10_x\in Ext_{A(2)}^{8,8+47}(A_1\sma DA_1)$ and $19_{56}(0_6)=19_y\in Ext_{A(2)}^{19,19+95}(A_1\sma DA_1)$, then $b_{3,0}^4\in Ext_{A(2)}^{8,48+8}(A_1\sma DA_1)$ supports a $d_2$ differential, and $b_{3,0}^8\in Ext_{A(2)}^{16,96+16}(A_1\sma DA_1)$ supports a $d_3$ differential. By doing the above steps for all four versions of $A_1$, and checking the respective \texttt{Map.aug} files, each contain lines
\begin{center}
\texttt{10 x 18}\\\texttt{19 y 56},
\end{center}
justifying the claim in Lemma~\ref{not8or16}.

Using the tools we have so far described, it is easy to verify the claim from the proof of Lemma~\ref{d3overA2}, that for all four models of $\A$ we have
\begin{equation} \label{eqn:lem4.1}
gb_{3,0}^4\cdot10_3 = 22_7.
\end{equation}
It is similarly easy to verify that if $\A=\A[00]$ or $\A=\A[11]$, we have\[ge_0r \cdot 10_3=0,\]while if $\A=\A[01]$ or $\A=\A[10]$, we have\[ge_0r\cdot10_3=24_0=g^6.\]

Finally, in order to run the algebraic $tmf$ spectral sequence, we will also need do do computations involving the $bo$-Brown-Gitler spectra. We give the $A$-module definitions for the cohomologies of $bo_1$ and $bo_2$ here:

\begin{figure}[H] 
\begin{subfigure}{.45\textwidth}
\begin{small}
\begin{verbatim}
4

0 4 6 7

0 4 1 1
0 6 1 2
0 7 1 3
1 2 1 2
1 3 1 3
2 1 1 3
\end{verbatim}
\end{small}
\caption{The text file \texttt{bo1\_def}}
\label{bo1-def}
\end{subfigure}%
\end{figure}
\begin{figure}
\begin{subfigure}{.45\textwidth}
\begin{small}
\begin{verbatim}
11

0 4 6 7 8 10 11 12 13 14 15

0 4 1 1 
0 6 1 2 
0 7 1 3 
1 2 1 2 
1 3 1 3 
2 1 1 3 
2 4 1 5 
2 5 1 6 
3 4 1 6 
3 6 1 8 
4 2 1 5 
4 3 1 6 
4 4 1 7 
4 5 1 8 
4 6 1 9 
4 7 1 10 
5 1 1 6 
6 2 1 8 
7 1 1 8 
7 2 1 9 
7 3 1 10 
9 1 1 10 
\end{verbatim}
\end{small}
\caption{The text file \texttt{bo2\_def}}
\label{bo2-def}
\end{subfigure}%
\end{figure} \clearpage

\section{Tables from section~\ref{three}}\label{appendixB}

\subsection{The cases $\A=\A[00]$ or $\A=\A[11]$}
\hspace{1cm}
\begin{table}[H] 
\caption{$Ext_{A(2)}^{s,t}(\A\sma D\A)$}\label{D:ADA1-00}
\[\begin{array}{|c|c|c||c|c|c||c|c|c|}
\hline
t-s&s&s_g&t-s&s&s_g&t-s&s&s_g\\
\hline
119&25&25_{0}&120&25&25_{2}&121&25&25_{5}\\
&&&120&25&25_{1}&121&25&25_{4}\\
&&&&&&121&25&25_{3}\\\hline
119&24&&120&24&24_{25}&121&24&24_{28}\\
&&&120&24&24_{24}&121&24&24_{27}\\
&&&120&24&24_{23}&121&24&24_{26}\\
&&&120&24&24_{22}&&&\\
&&&120&24&24_{21}&&&\\\hline
119&23&&120&23&23_{42}&121&23&23_{47}\\
&&&120&23&23_{41}&121&23&23_{46}\\
&&&120&23&23_{40}&121&23&23_{45}\\
&&&120&23&23_{39}&121&23&23_{44}\\
&&&120&23&23_{38}&121&23&23_{43}\\\hline
119&22&&120&22&22_{64}&121&22&22_{68}\\
&&&120&22&22_{63}&121&22&22_{67}\\
&&&120&22&22_{62}&121&22&22_{66}\\
&&&&&&121&22&22_{65}\\\hline
119&21&&120&21&&121&21&21_{91}\\
&&&&&&121&21&21_{90}\\\hline
\end{array}\]
\[\begin{array}{|c|c|c||c|c|c|}
\hline
t-s&s&s_g&t-s&s&s_g\\
\hline
24&6&6_{1}&25&6&6_{4}\\
&&&25&6&6_{3}\\
&&&25&6&6_{2}\\\hline
24&5&5_{25}&25&5&5_{27}\\
24&5&5_{24}&25&5&5_{26}\\
24&5&5_{23}&&&\\
24&5&5_{22}&&&\\\hline
\end{array}\]
\[\begin{array}{|c|c|}
\hline
x&b_{3,0}^8x\\\hline
5_{27}&21_{91}\\
5_{26}&21_{90}\\\hline
\end{array}\]
\end{table}
\newpage
\begin{table}[H] 
\caption{$Ext_{A(2)}^{s-1,t-8}(\A\sma D\A\sma bo_1)$}\label{D:ADA1-00bo1}
\[
\begin{array}{|c|c|c||c|c|c||c|c|c|}
\hline
t-s&s&s_g&t-s&s&s_g&t-s&s&s_g\\
\hline
119&25&&120&25&&121&25&\\
\hline
119&24&&120&24&&121&24&\\
\hline
119&23&22_{5}&120&23&22_{11}&121&23&22_{19}\\
119&23&22_{4}&120&23&22_{10}&121&23&22_{18}\\
119&23&22_{3}&120&23&22_{9}&121&23&22_{17}\\
119&23&22_{2}&120&23&22_{8}&121&23&22_{16}\\
&&&120&23&22_{7}&121&23&22_{15}\\
&&&120&23&22_{6}&121&23&22_{14}\\
&&&&&&121&23&22_{13}\\
&&&&&&121&23&22_{12}\\\hline
119&22&21_{31}&120&22&&121&22&21_{33}\\
119&22&21_{30}&120&22&&121&22&21_{32}\\\hline
119&21&20_{51}&120&21&20_{57}&121&21&20_{62}\\
119&21&20_{50}&120&21&20_{56}&121&21&20_{61}\\
119&21&20_{49}&120&21&20_{55}&121&21&20_{60}\\
119&21&20_{48}&120&21&20_{54}&121&21&20_{59}\\
119&21&20_{47}&120&21&20_{53}&121&21&20_{58}\\
119&21&20_{46}&120&21&20_{52}&&&\\
119&21&20_{45}&&&&&&\\
119&21&20_{44}&&&&&&\\\hline
\end{array}\]
\[\begin{array}{|c|c|c||c|c|c|}
\hline
t-s&s&s_g&t-s&s&s_g\\
\hline
24&6&&25&6&\\\hline
24&5&&25&5&4_{0}\\\hline
\end{array}\]
\end{table}
\begin{table}[H] 
\caption{$Ext_{A(2)}^{s-1,t-16}(\A\sma D\A\sma bo_2)$}\label{D:ADA1-00bo2}
\[
\begin{array}{|c|c|c||c|c|c||c|c|c|}
\hline
t-s&s&s_g&t-s&s&s_g&t-s&s&s_g\\
\hline
119&25&&120&25&&121&25&\\
\hline
119&24&&120&24&&121&24&\\
\hline
119&23&&120&23&&121&23&\\
\hline
119&22&&120&22&&121&22&\\
\hline
119&21&&120&21&20_{0}&121&21&20_{2}\\
&&&&&&121&21&20_{1}\\\hline
\end{array}\]
\[\begin{array}{|c|c|c||c|c|c|}
\hline
t-s&s&s_g&t-s&s&s_g\\
\hline
24&6&&25&6&\\\hline
24&5&&25&5&\\\hline
\end{array}\]
\end{table}
\newpage
\begin{table}[H] 
\caption{$Ext_{A(2)}^{s-2,t-16}(\A\sma D\A\sma bo_1\sma bo_1)$}\label{D:ADA1-00bo1bo1}
\[\begin{array}{|c|c|c||c|c|c||c|c|c|}
\hline
t-s&s&s_g&t-s&s&s_g&t-s&s&s_g\\
\hline
119&25&&120&25&&121&25&\\
\hline
119&24&&120&24&&121&24&\\
\hline
119&23&&120&23&&121&23&\\
\hline
119&22&20_{1}&120&22&20_{5}&121&22&20_{11}\\
119&22&20_{0}&120&22&20_{4}&121&22&20_{10}\\
&&&120&22&20_{3}&121&22&20_{9}\\
&&&120&22&20_{2}&121&22&20_{8}\\
&&&&&&121&22&20_{7}\\
&&&&&&121&22&20_{6}\\\hline
119&21&19_{41}&120&21&19_{51}&121&21&19_{57}\\
119&21&19_{40}&120&21&19_{50}&121&21&19_{56}\\
119&21&19_{39}&120&21&19_{49}&121&21&19_{55}\\
119&21&19_{38}&120&21&19_{48}&121&21&19_{54}\\
119&21&19_{37}&120&21&19_{47}&121&21&19_{53}\\
119&21&19_{36}&120&21&19_{46}&121&21&19_{52}\\
119&21&19_{35}&120&21&19_{45}&&&\\
119&21&19_{34}&120&21&19_{44}&&&\\
119&21&19_{33}&120&21&19_{43}&&&\\
119&21&19_{32}&120&21&19_{42}&&&\\\hline
\end{array}\]
\[\begin{array}{|c|c|c||c|c|c|}
\hline
t-s&s&s_g&t-s&s&s_g\\
\hline
24&6&&25&6&\\\hline
24&5&&25&5&\\\hline
\end{array}\]
\end{table}
\clearpage

\subsection{The case $\A=\A[01]$}
\hspace{1cm}
\begin{table}[H] 
\caption{$Ext_{A(2)}^{s,t}(\A\sma D\A)$}\label{D:ADA1-01}
\[\begin{array}{|c|c|c||c|c|c||c|c|c|}
\hline
t-s&s&s_g&t-s&s&s_g&t-s&s&s_g\\
\hline
119&25&&120&25&&121&25&25_{0}\\\hline
119&24&24_{14}&120&24&24_{18}&121&24&24_{19}\\
119&24&24_{13}&120&24&24_{17}&&&\\
119&24&24_{12}&120&24&24_{16}&&&\\
119&24&24_{11}&120&24&24_{15}&&&\\
119&24&24_{10}&&&&&&\\\hline
119&23&23_{20}&120&23&23_{24}&121&23&23_{29}\\
&&&120&23&23_{23}&121&23&23_{28}\\
&&&120&23&23_{22}&121&23&23_{27}\\
&&&120&23&23_{21}&121&23&23_{26}\\
&&&&&&121&23&23_{25}\\\hline
119&22&22_{41}&120&22&&121&22&22_{41}\\
119&22&22_{40}&&&&&&\\\hline
119&21&21_{60}&120&21&21_{64}&121&21&21_{65}\\
119&21&21_{59}&120&21&21_{63}&&&\\
119&21&21_{58}&120&21&21_{62}&&&\\
119&21&21_{57}&120&21&21_{61}&&&\\
119&21&21_{56}&&&&&&\\
119&21&21_{55}&&&&&&\\
119&21&21_{54}&&&&&&\\\hline
\end{array}\]
\[\begin{array}{|c|c|c||c|c|c|}
\hline
t-s&s&s_g&t-s&s&s_g\\
\hline
24&6&&25&6&6_{1}\\\hline
24&5&5_{24}&25&5&5_{25}\\
24&5&5_{23}&&&\\
24&5&5_{22}&&&\\
24&5&5_{21}&&&\\\hline
\end{array}\]
\[\begin{array}{|c|c|}
\hline
x&b_{3,0}^8x\\\hline
5_{25}&21_{65}\\\hline
\end{array}\]
\end{table}
\newpage
\begin{table}[H] 
\caption{$Ext_{A(2)}^{s-1,t-8}(\A\sma D\A\sma bo_1)$}\label{D:ADA1-01bo1}
\[
\begin{array}{|c|c|c||c|c|c||c|c|c|}
\hline
t-s&s&s_g&t-s&s&s_g&t-s&s&s_g\\
\hline
119&25&&120&25&&121&25&\\
\hline
119&24&&120&24&&121&24&\\
\hline
119&23&22_{3}&120&23&22_{9}&121&23&22_{17}\\
119&23&22_{2}&120&23&22_{8}&121&23&22_{16}\\
119&23&22_{1}&120&23&22_{7}&121&23&22_{15}\\
119&23&22_{0}&120&23&22_{6}&121&23&22_{14}\\
&&&120&23&22_{5}&121&23&22_{13}\\
&&&120&23&22_{4}&121&23&22_{12}\\
&&&&&&121&23&22_{11}\\
&&&&&&121&23&22_{10}\\\hline
119&22&&120&22&&121&22&\\\hline
119&21&20_{45}&120&21&20_{51}&121&21&20_{56}\\
119&21&20_{44}&120&21&20_{50}&121&21&20_{55}\\
119&21&20_{43}&120&21&20_{49}&121&21&20_{54}\\
119&21&20_{42}&120&21&20_{48}&121&21&20_{53}\\
119&21&20_{41}&120&21&20_{47}&121&21&20_{52}\\
119&21&20_{40}&120&21&20_{46}&&&\\
119&21&20_{39}&&&&&&\\
119&21&20_{38}&&&&&&\\\hline
\end{array}\]
\[\begin{array}{|c|c|c||c|c|c|}
\hline
t-s&s&s_g&t-s&s&s_g\\
\hline
24&6&&25&6&\\\hline
24&5&&25&5&4_{0}\\\hline
\end{array}\]
\end{table}
\begin{table}[H] 
\caption{$Ext_{A(2)}^{s-1,t-16}(\A\sma D\A\sma bo_2)$}\label{D:ADA1-01bo2}
\[
\begin{array}{|c|c|c||c|c|c||c|c|c|}
\hline
t-s&s&s_g&t-s&s&s_g&t-s&s&s_g\\
\hline
119&25&&120&25&&121&25&\\
\hline
119&24&&120&24&&121&24&\\
\hline
119&23&&120&23&&121&23&\\
\hline
119&22&&120&22&&121&22&\\
\hline
119&21&&120&21&&121&21&\\\hline
\end{array}\]
\[\begin{array}{|c|c|c||c|c|c|}
\hline
t-s&s&s_g&t-s&s&s_g\\
\hline
24&6&&25&6&\\\hline
24&5&&25&5&\\\hline
\end{array}\]
\end{table}
\newpage
\begin{table}[H] 
\caption{$Ext_{A(2)}^{s-2,t-16}(\A\sma D\A\sma bo_1\sma bo_1)$}\label{D:ADA1-01bo1bo1}
\[\begin{array}{|c|c|c||c|c|c||c|c|c|}
\hline
t-s&s&s_g&t-s&s&s_g&t-s&s&s_g\\
\hline
119&25&&120&25&&121&25&\\
\hline
119&24&&120&24&&121&24&\\
\hline
119&23&&120F&23&&121&23&\\
\hline
119&22&&120&22&&121&22&20_{1}\\
&&&&&&121&22&20_{0}\\\hline
119&21&19_{29}&120&21&19_{37}&121&21&19_{39}\\
119&21&19_{28}&120&21&19_{36}&121&21&19_{38}\\
119&21&19_{27}&120&21&19_{35}&&&\\
119&21&19_{26}&120&21&19_{34}&&&\\
119&21&19_{25}&120&21&19_{33}&&&\\
119&21&19_{24}&120&21&19_{32}&&&\\
119&21&19_{23}&120&21&19_{31}&&&\\
119&21&19_{22}&120&21&19_{30}&&&\\
119&21&19_{21}&&&&&&\\
119&21&19_{20}&&&&&&\\\hline
\end{array}\]
\[\begin{array}{|c|c|c||c|c|c|}
\hline
t-s&s&s_g&t-s&s&s_g\\
\hline
24&6&&25&6&\\\hline
24&5&&25&5&\\\hline
\end{array}\]
\end{table}
\clearpage
\subsection{The case $\A=\A[10]$}
\hspace{1cm}
\begin{table}[H] 
\caption{$Ext_{A(2)}^{s,t}(\A\sma D\A)$}\label{D:ADA1-10}
\[\begin{array}{|c|c|c||c|c|c||c|c|c|}
\hline
t-s&s&s_g&t-s&s&s_g&t-s&s&s_g\\
\hline
119&25&&120&25&&121&25&25_{0}\\\hline
119&24&24_{14}&120&24&24_{18}&121&24&24_{19}\\
119&24&24_{13}&120&24&24_{17}&&&\\
119&24&24_{12}&120&24&24_{16}&&&\\
119&24&24_{11}&120&24&24_{15}&&&\\
119&24&24_{10}&&&&&&\\\hline
119&23&23_{20}&120&23&23_{24}&121&23&23_{29}\\
&&&120&23&23_{23}&121&23&23_{28}\\
&&&120&23&23_{22}&121&23&23_{27}\\
&&&120&23&23_{21}&121&23&23_{26}\\
&&&&&&121&23&23_{25}\\\hline
119&22&22_{40}&120&22&&121&22&22_{41}\\
119&22&22_{39}&&&&&&\\\hline
119&21&21_{61}&120&21&21_{65}&121&21&21_{66}\\
119&21&21_{60}&120&21&21_{64}&&&\\
119&21&21_{59}&120&21&21_{63}&&&\\
119&21&21_{58}&120&21&21_{62}&&&\\
119&21&21_{57}&&&&&&\\
119&21&21_{56}&&&&&&\\
119&21&21_{55}&&&&&&\\\hline
\end{array}\]
\[\begin{array}{|c|c|c||c|c|c|}
\hline
t-s&s&s_g&t-s&s&s_g\\
\hline
24&6&&25&6&6_{1}\\\hline
24&5&5_{25}&25&5&5_{26}\\
24&5&5_{24}&&&\\
24&5&5_{23}&&&\\
24&5&5_{22}&&&\\\hline
\end{array}\]
\[\begin{array}{|c|c|}
\hline
x&b_{3,0}^8x\\\hline
5_{26}&21_{66}\\\hline
\end{array}\]
\end{table}
\newpage
\begin{table}[H] 
\caption{$Ext_{A(2)}^{s-1,t-8}(\A\sma D\A\sma bo_1)$}\label{D:ADA1-10bo1}
\[
\begin{array}{|c|c|c||c|c|c||c|c|c|}
\hline
t-s&s&s_g&t-s&s&s_g&t-s&s&s_g\\
\hline
119&25&&120&25&&121&25&\\
\hline
119&24&&120&24&&121&24&\\
\hline
119&23&22_{3}&120&23&22_{9}&121&23&22_{17}\\
119&23&22_{2}&120&23&22_{8}&121&23&22_{16}\\
119&23&22_{1}&120&23&22_{7}&121&23&22_{15}\\
119&23&22_{0}&120&23&22_{6}&121&23&22_{14}\\
&&&120&23&22_{5}&121&23&22_{13}\\
&&&120&23&22_{4}&121&23&22_{12}\\
&&&&&&121&23&22_{11}\\
&&&&&&121&23&22_{10}\\\hline
119&22&&120&22&&121&22&\\\hline
119&21&20_{45}&120&21&20_{51}&121&21&20_{56}\\
119&21&20_{44}&120&21&20_{50}&121&21&20_{55}\\
119&21&20_{43}&120&21&20_{49}&121&21&20_{54}\\
119&21&20_{42}&120&21&20_{48}&121&21&20_{53}\\
119&21&20_{41}&120&21&20_{47}&121&21&20_{52}\\
119&21&20_{40}&120&21&20_{46}&&&\\
119&21&20_{39}&&&&&&\\
119&21&20_{38}&&&&&&\\\hline
\end{array}\]
\[\begin{array}{|c|c|c||c|c|c|}
\hline
t-s&s&s_g&t-s&s&s_g\\
\hline
24&6&&25&6&\\\hline
24&5&&25&5&4_{0}\\\hline
\end{array}\]
\end{table}

\begin{table}[H] 
\caption{$Ext_{A(2)}^{s-1,t-16}(\A\sma D\A\sma bo_2)$}\label{D:ADA1-10bo2}
\[
\begin{array}{|c|c|c||c|c|c||c|c|c|}
\hline
t-s&s&s_g&t-s&s&s_g&t-s&s&s_g\\
\hline
119&25&&120&25&&121&25&\\
\hline
119&24&&120&24&&121&24&\\
\hline
119&23&&120&23&&121&23&\\
\hline
119&22&&120&22&&121&22&\\
\hline
119&21&&120&21&&121&21&\\
\hline
\end{array}\]
\[\begin{array}{|c|c|c||c|c|c|}
\hline
t-s&s&s_g&t-s&s&s_g\\
\hline
24&6&&25&6&\\\hline
24&5&&25&5&\\\hline
\end{array}\]
\end{table}
\newpage
\begin{table}[H] 
\caption{$Ext_{A(2)}^{s-2,t-16}(\A\sma D\A\sma bo_1\sma bo_1)$}\label{D:ADA1-10bo1bo1}
\[\begin{array}{|c|c|c||c|c|c||c|c|c|}
\hline
t-s&s&s_g&t-s&s&s_g&t-s&s&s_g\\
\hline
119&25&&120&25&&121&25&\\
\hline
119&24&&120&24&&121&24&\\
\hline
119&23&&120&23&&121&23&\\
\hline
119&22&&120&22&&121&22&20_{1}\\
&&&&&&121&22&20_{0}\\\hline
119&21&19_{29}&120&21&19_{37}&121&21&19_{39}\\
119&21&19_{28}&120&21&19_{36}&121&21&19_{38}\\
119&21&19_{27}&120&21&19_{35}&&&\\
119&21&19_{26}&120&21&19_{34}&&&\\
119&21&19_{25}&120&21&19_{33}&&&\\
119&21&19_{24}&120&21&19_{32}&&&\\
119&21&19_{23}&120&21&19_{31}&&&\\
119&21&19_{22}&120&21&19_{30}&&&\\
119&21&19_{21}&&&&&&\\
119&21&19_{20}&&&&&&\\\hline
\end{array}\]
\[\begin{array}{|c|c|c||c|c|c|}
\hline
t-s&s&s_g&t-s&s&s_g\\
\hline
24&6&&25&6&\\\hline
24&5&&25&5&\\\hline
\end{array}\]
\end{table}
\clearpage

\section{Tables from section~\ref{four}}\label{appendixC}
\subsection{The case $\A=\A[00]$ or $\A=\A[11]$}
\hspace{1cm}
\begin{table}[H] 
\caption{$Ext_{A(2)}^{s,t}(\A\sma D\A)$}\label{ADA1-00}
\[
\begin{array}{|c|c|c||c|c|c||||c|c|c||c|c|c|}
\hline
t-s&s&s_g&t-s&s&s_g&t-s&s&s_g&t-s&s&s_g\\
\hline
\hline
70&15&15_{2}&71 & 15 & 15_{5}&190&39&39_{2}&191 & 39 & 39_{5}\\
70&15&15_{1} &71 & 15 & 15_{4}&190&39&39_{1}&191 & 39 & 39_{4}\\
  & & &71 & 15 & 15_{3} & & & &191 & 39 & 39_{3}\\
\hline
70 & 14 & 14_{25} &71 & 14 & 14_{29}&190& 38 & 38_{25}&191 & 38 & 38_{28}\\
70 & 14 & 14_{24} &71 & 14 & 14_{28}&190& 38 & 38_{24}&191 & 38 & 38_{27}\\
70 & 14 & 14_{23} &71 & 14 & 14_{27}&190& 38 & 38_{23}&191 & 38 & 38_{26}\\
70 & 14 & 14_{22} &71 & 14 & 14_{26}&190& 38 & 38_{22}&&&\\
70 & 14 & 14_{21} & & &&190& 38 & 38_{21}&&&\\
\hline
70 & 13 & 13_{46} &71 & 13 & 13_{53}&190& 37 & 37_{42}&191 & 37 & 37_{47}\\
70 & 13 & 13_{45} &71 & 13 & 13_{52}&190& 37 & 37_{41}&191 & 37 & 37_{46}\\
70 & 13 & 13_{44} &71 & 13 & 13_{51}&190& 37 & 37_{40}&191 & 37 & 37_{45}\\
70 & 13 & 13_{43} &71 & 13 & 13_{50}&190& 37 & 37_{39}&191 & 37 & 37_{44}\\
70 & 13 & 13_{42} &71 & 13 & 13_{49}&190& 37 & 37_{38}&191 & 37 & 37_{43}\\
&&&71 & 13 & 13_{48}&&&&&&\\
&&&71 & 13 & 13_{47}&&&&&&\\
\hline
 & & &  & & & 190& 36 & 36_{63}&191 & 36 & 36_{66}\\
 & & &  & & & 190& 36 & 36_{62}&191 & 36 & 36_{65}\\
 & & &  & & & 190& 36 & 36_{61}&191 & 36 & 36_{64}\\
\hline
\end{array}\]
\[\begin{array}{|c|c|c|c|c|}\hline
n&i_1,\ldots,i_n&x & g^6x & v_2^{20}h_1x\\\hline
0&0&15_5&39_5&36_{69}\\
0&0&15_4&39_4&36_{68}\\
0&0&15_3&39_3&36_{67}\\
0&0&15_2&39_2&36_{66}\\
0&0&15_1&39_1&36_{65}\\\hline
0&0&14_{29}&0&0\\
0&0&14_{28}&38_{28}&35_{92}\\
0&0&14_{27}&38_{27}&35_{91}\\
0&0&14_{26}&38_{26}&35_{90}\\
0&0&14_{25}&38_{25}&35_{89}\\
0&0&14_{24}&38_{24}&35_{88}\\
0&0&14_{23}&38_{23}&35_{87}\\
0&0&14_{22}&38_{22}&35_{86}\\
0&0&14_{21}&38_{21}&35_{85}\\\hline
0&0&13_{53}&0&0\\
0&0&13_{52}&0&0\\
0&0&13_{51}&37_{44}&34_{108}\\
0&0&13_{50}&37_{43}&34_{107}\\
0&0&13_{49}&37_{43}+37_{45}&34_{107}+34_{109}\\
0&0&13_{48}&37_{45}+37_{46}+37_{47}&34_{109}+34_{110}+34_{111}\\
0&0&13_{47}&37_{43}+37_{45}+37_{46}&34_{107}+34_{109}+34_{110}\\\hline
\end{array}\]
\end{table}

\begin{table}[H] 
\caption{$Ext_{A(2)}^{s-1,t-8}(\A\sma D\A\sma bo_1)$}\label{ADA1-00bo1}
\[
\begin{array}{|c|c|c||c|c|c||||c|c|c||c|c|c|}
\hline
t-s&s&s_g&t-s&s&s_g&t-s&s&s_g&t-s&s&s_g\\
\hline
\hline
70&15&&71&15&&190&39&&191&39&\\
\hline
70&14&&71&14&&190&38&&191&38&\\
\hline
70&13 & 12_{11} &71 & 13 & 12_{19}&190 & 37 & 36_{11} &191 & 37 & 36_{19}\\
70&13 & 12_{10} &71 & 13 & 12_{18}&190 & 37 & 36_{10} &191 & 37 & 36_{18}\\
70&13 & 12_{9} &71 & 13 & 12_{17}&190 & 37 & 36_{9} &191 & 37 & 36_{17}\\
70&13 & 12_{8} &71 & 13 & 12_{16}&190 & 37 & 36_{8} &191 & 37 & 36_{16}\\
70&13 & 12_{7} &71 & 13 & 12_{15}&190 & 37 & 36_{7} &191 & 37 & 36_{15}\\
70&13 & 12_{6} &71 & 13 & 12_{14}&190 & 37 & 36_{6} &191 & 37 & 36_{14}\\
&&&71 & 13 & 12_{13}&&&&191 & 37 & 36_{13}\\
&&&71 & 13 & 12_{12}&&&&191 & 37 & 36_{12}\\
\hline
70 & 12 & 11_{40} &71 & 12 & 11_{46}&190&36&&191 & 36 & 35_{33}\\
70 & 12 & 11_{39} &71 & 12 & 11_{45}&&&&191 & 36 & 35_{32}\\
70 & 12 & 11_{38} &71 & 12 & 11_{44}&&&&&&\\
70 & 12 & 11_{37} &71 & 12 & 11_{43}&&&&&&\\
70 & 12 & 11_{36} &71 & 12 & 11_{42}&&&&&&\\
70 & 12 & 11_{35} &71 & 12 & 11_{41}&&&&&&\\
70 & 12 & 11_{34} &&&&&&&&&\\
\hline
\end{array}\]
\[\begin{array}{|c|c|c|c|c|}\hline
n&i_1,\ldots,i_n&x&g^6x&v_2^{20}h_1x\\\hline
1&1&12_{19}&36_{19}&33_{83}\\
1&1&12_{18}&36_{18}&33_{82}\\
1&1&12_{17}&36_{17}&33_{79}+33_{83}\\
1&1&12_{16}&36_{16}&33_{79}+33_{81}\\
1&1&12_{15}&36_{15}&33_{80}\\
1&1&12_{14}&36_{14}&33_{78}+33_{79}+33_{81}+33_{83}\\
1&1&12_{13}&36_{13}&33_{77}\\
1&1&12_{12}&36_{12}&33_{76}+33_{83}\\\hline
1&1&11_{46}&0&0\\
1&1&11_{45}&0&0\\
1&1&11_{44}&0&0\\
1&1&11_{43}&35_{33}&32_{97}\\
1&1&11_{42}&35_{32}&32_{96}\\
1&1&11_{41}&0&0\\\hline
\end{array}\]
\end{table}

\begin{table}[H] 
\caption{$Ext_{A(2)}^{s-1,t-16}(\A\sma D\A\sma bo_2)$}\label{ADA1-00bo2}
\[
\begin{array}{|c|c|c||c|c|c||||c|c|c||c|c|c|}
\hline
t-s&s&s_g&t-s&s&s_g&t-s&s&s_g&t-s&s&s_g\\
\hline
\hline
70&15&&71&15&&190&39&&191&39&\\
\hline
70&14&&71&14&&190&38&&191&38&\\
\hline
70&13&&71&13&&190&37&&191&37&\\
\hline
70&12&&71&12&&190&36&&191&36 &\\
\hline
\end{array}\]
\end{table}

\begin{table}[H] 
\caption{$Ext_{A(2)}^{s-2,t-16}(\A\sma D\A\sma bo_1\sma bo_1)$}\label{ADA1-00bo1bo1}
\[
\begin{array}{|c|c|c||c|c|c||||c|c|c||c|c|c|}
\hline
t-s&s&s_g&t-s&s&s_g&t-s&s&s_g&t-s&s&s_g\\
\hline
\hline
70&15&&71&15&&190&39&&191&39&\\
\hline
70&14&&71&14&&190&38&&191&38&\\
\hline
70&13&&71&13&&190&37&&191&37&\\
\hline
70 & 12 & 10_{5} &71 & 12 & 10_{11} &190 & 36 & 34_{5}&191 & 36 & 34_{11}\\
70 & 12 & 10_{4} &71 & 12 & 10_{10} &190 & 36 & 34_{4}&191 & 36 & 34_{10}\\
70 & 12 & 10_{3} &71 & 12 & 10_{9} &190 & 36 & 34_{3}&191 & 36 & 34_{9}\\
70 & 12 & 10_{2} &71 & 12 & 10_{8} &190 & 36 & 34_{2}&191 & 36 & 34_{8}\\
&&&71 & 12 & 10_{7} &&&&191 & 36 & 34_{7}\\
&&&71 & 12 & 10_{6} &&&&191 & 36 & 34_{6}\\
\hline
\end{array}\]
\[\begin{array}{|c|c|c|c|c|}\hline
n&i_1,\ldots,i_n&x & g^6x & v_2^{20}h_1x\\\hline
2&1,1&10_{11}&34_{11}&31_{139}\\
2&1,1&10_{10}&34_{10}&31_{138}\\
2&1,1&10_{9}&34_9&31_{137}\\
2&1,1&10_{8}&34_8&31_{136}\\
2&1,1&10_{7}&34_7&31_{135}\\
2&1,1&10_{6}&34_6&31_{134}+31_{137}+31_{138}\\
\hline
\end{array}\]
\end{table}

\begin{table}[H] 
\caption{$Ext_{A(2)}^{s,t}(\A\sma D\A)$}\label{Lemma5.3:ADA1-00}
\[
\begin{array}{|c|c|c||c|c|c||c|c|c|}
\hline
t-s&s&s_g&t-s&s&s_g&t-s&s&s_g\\
\hline
\hline
95&20&20_2&143&28&28_{34}&191&36&36_{66}\\
95&20&20_1&143&28&28_{33}&191&36&36_{65}\\
&&&143&28&28_{32}&191&36&36_{64}\\
\hline
\end{array}\]
\[\begin{array}{|c|c|c|c|c|}\hline
x & b_{3,0}^4\cdot x & e_0r\cdot x&b_{3,0}^8\cdot x&wgr\cdot x\\\hline
28_{34}&36_{66}&0&N/A&N/A\\
28_{33}&36_{65}&0&N/A&N/A\\
28_{32}&36_{64}&38_{25}&N/A&N/A\\
20_2&28_{34}&N/A&36_{66}&39_2\\
20_1&28_{33}&N/A&36_{65}&39_1\\
\hline
\end{array}\]
\end{table}

\clearpage
\subsection{The case $\A=\A[01]$}
\hspace{1cm}
\begin{table}[H] 
\caption{$Ext_{A(2)}^{s,t}(\A\sma D\A)$}\label{ADA1-01}
\[
\begin{array}{|c|c|c||c|c|c||||c|c|c||c|c|c|}
\hline
t-s&s&s_g&t-s&s&s_g&t-s&s&s_g&t-s&s&s_g\\
\hline
\hline
70&15&&71 & 15 & 15_{0}&190&39&&191 & 39 & 39_{0}\\
\hline
70 & 14 & 14_{18}&71 & 14 & 14_{20}&190 & 38 & 38_{18}&191 & 38 & 38_{19}\\
70 & 14 & 14_{17}&71 & 14 & 14_{19}&190 & 38 & 38_{17}&&&\\
70 & 14 & 14_{16}&&&&190 & 38 & 38_{16}&&&\\
70 & 14 & 14_{15}&&&&190 & 38 & 38_{15}&&&\\
\hline
70 & 13 & 13_{33}&71 & 13 & 13_{40}&190 & 37 & 37_{24}&191 & 37 & 37_{29}\\
70 & 13 & 13_{32}&71 & 13 & 13_{39}&190 & 37 & 37_{23}&191 & 37 & 37_{28}\\
70 & 13 & 13_{31}&71 & 13 & 13_{38}&190 & 37 & 37_{22}&191 & 37 & 37_{27}\\
70 & 13 & 13_{30}&71 & 13 & 13_{37}&190 & 37 & 37_{21}&191 & 37 & 37_{26}\\
70 & 13 & 13_{29}&71 & 13 & 13_{36}&&&&191 & 37 & 37_{25}\\
&&&71 & 13 & 13_{35}&&&&&&\\
&&&71 & 13 & 13_{34}&&&&&&\\
\hline
70&12&&71&12&&190 & 36 & &191 & 36 & \\
\hline
\end{array}\]
\[\begin{array}{|c|c|c|}\hline
x&g^6x&v_2^{20}h_1x\\\hline
15_0&39_0&36_{40}\\\hline
14_{20}&0&0\\
14_{19}&38_{19}&35_{59}\\
14_{18}&38_{18}&35_{58}\\
14_{17}&38_{17}&35_{57}+35_{58}\\
14_{16}&38_{16}&35_{56}\\
14_{15}&38_{15}&35_{55}+35_{58}\\\hline
13_{40}&0&0\\
13_{39}&0&0\\
13_{38}&37_{29}&34_{69}\\
13_{37}&37_{28}&34_{68}\\
13_{36}&37_{27}&34_{67}+34_{68}\\
13_{35}&37_{26}&34_{66}+34_{67}+34_{69}\\
13_{34}&37_{25}&34_{65}+34_{66}\\
\hline
\end{array}\]
\end{table}
\newpage
\begin{table}[H] 
\caption{$Ext_{A(2)}^{s-1,t-8}(\A\sma D\A\sma bo_1)$}\label{ADA1-01bo1}
\[
\begin{array}{|c|c|c||c|c|c||||c|c|c||c|c|c|}
\hline
t-s&s&s_g&t-s&s&s_g&t-s&s&s_g&t-s&s&s_g\\
\hline
\hline
70&15&&71&15&&190&39&&191&39&\\
\hline
70&14&&71&14&&190&38&&191&38&\\
\hline
70&13 & 12_{9} &71 & 13 & 12_{17}&190 & 37 & 36_{9} &191 & 37 & 36_{17}\\
70&13 & 12_{8} &71 & 13 & 12_{16}&190 & 37 & 36_{8} &191 & 37 & 36_{16}\\
70&13 & 12_{7} &71 & 13 & 12_{15}&190 & 37 & 36_{7} &191 & 37 & 36_{15}\\
70&13 & 12_{6} &71 & 13 & 12_{14}&190 & 37 & 36_{6} &191 & 37 & 36_{14}\\
70&13 & 12_{5} &71 & 13 & 12_{13}&190 & 37 & 36_{5} &191 & 37 & 36_{13}\\
70&13 & 12_{4} &71 & 13 & 12_{12}&190 & 37 & 36_{4} &191 & 37 & 36_{12}\\
&&&71 & 13 & 12_{11}&&&&191 & 37 & 36_{11}\\
&&&71 & 13 & 12_{10}&&&&191 & 37 & 36_{10}\\
\hline
70 & 12 & 11_{36} &71 & 12 & 11_{42}&190&36&&191&36&\\
70 & 12 & 11_{35} &71 & 12 & 11_{41}&&&&&&\\
70 & 12 & 11_{34} &71 & 12 & 11_{40}&&&&&&\\
70 & 12 & 11_{33} &71 & 12 & 11_{39}&&&&&&\\
70 & 12 & 11_{32} &71 & 12 & 11_{38}&&&&&&\\
70 & 12 & 11_{31} &71 & 12 & 11_{37}&&&&&&\\
70 & 12 & 11_{30} &&&&&&&&&\\
\hline
\end{array}\]
\[\begin{array}{|c|c|c|}\hline
x&g^6x&v_2^{20}h_1x\\\hline
12_{17}&36_{17}&33_{73}\\
12_{16}&36_{16}&33_{72}+33_{73}\\
12_{15}&36_{15}&33_{71}\\
12_{14}&36_{14}&33_{70}+33_{71}\\
12_{13}&36_{13}&33_{69}+33_{71}+33_{72}+33_{73}\\
12_{12}&36_{12}&33_{68}+33_{69}+33_{71}+33_{72}+33_{73}\\
12_{11}&36_{11}&33_{67}+33_{68}+33_{69}+33_{72}\\
12_{10}&36_{10}&33_{66}+33_{72}\\\hline
\end{array}\]
\end{table}

\begin{table}[H] 
\caption{$Ext_{A(2)}^{s-1,t-16}(\A\sma D\A\sma bo_2)$}\label{ADA1-01bo2}
\[
\begin{array}{|c|c|c||c|c|c||||c|c|c||c|c|c|}
\hline
t-s&s&s_g&t-s&s&s_g&t-s&s&s_g&t-s&s&s_g\\
\hline
\hline
70&15&&71&15&&190&39&&191&39&\\
\hline
70&14&&71&14&&190&38&&191&38&\\
\hline
70&13&&71&13&&190&37&&191&37&\\
\hline
70&12&&71&12&&190&36&&191&36&\\\hline
\end{array}\]
\end{table}

\begin{table}[H] 
\caption{$Ext_{A(2)}^{s-2,t-16}(\A\sma D\A\sma bo_1\sma bo_1)$}\label{ADA1-01bo1bo1}
\[
\begin{array}{|c|c|c||c|c|c||||c|c|c||c|c|c|}
\hline
t-s&s&s_g&t-s&s&s_g&t-s&s&s_g&t-s&s&s_g\\
\hline
\hline
70&15&&71&15&&190&39&&191&39&\\
\hline
70&14&&71&14&&190&38&&191&38&\\
\hline
70&13&&71&13&&190&37&&191&37&\\
\hline
70&12&&71&12&10_1&190&36&&191&36&34_1\\
&&&71&12&10_0&&&&191&36&34_0\\\hline
\end{array}\]
\[\begin{array}{|c|c|c|}\hline
x&g^6x&v_2^{20}h_1x\\\hline
10_{1}&34_{1}&31_{81}\\
10_{0}&34_{0}&31_{80}\\\hline
\end{array}\]
\end{table}

%

\subsection{The case $\A=\A[10]$}
\hspace{1cm}
\begin{table}[H] 
\caption{$Ext_{A(2)}^{s,t}(\A\sma D\A)$}\label{ADA1-10}
\[
\begin{array}{|c|c|c||c|c|c||||c|c|c||c|c|c|}
\hline
t-s&s&s_g&t-s&s&s_g&t-s&s&s_g&t-s&s&s_g\\
\hline
\hline
&&&71 & 15 & 15_{0}&&&&191 & 39 & 39_{0}\\
\hline
70 & 14 & 14_{18}&71 & 14 & 14_{20}&190 & 38 & 38_{18}&191 & 38 & 38_{19}\\
70 & 14 & 14_{17}&71 & 14 & 14_{19}&190 & 38 & 38_{17}&&&\\
70 & 14 & 14_{16}&&&&190 & 38 & 38_{16}&&&\\
70 & 14 & 14_{15}&&&&190 & 38 & 38_{15}&&&\\
\hline
70 & 13 & 13_{34}&71 & 13 & 13_{41}&190 & 37 & 37_{24}&191 & 37 & 37_{29}\\
70 & 13 & 13_{33}&71 & 13 & 13_{40}&190 & 37 & 37_{23}&191 & 37 & 37_{28}\\
70 & 13 & 13_{32}&71 & 13 & 13_{39}&190 & 37 & 37_{22}&191 & 37 & 37_{27}\\
70 & 13 & 13_{31}&71 & 13 & 13_{38}&190 & 37 & 37_{21}&191 & 37 & 37_{26}\\
70 & 13 & 13_{30}&71 & 13 & 13_{37}&&&&191 & 37 & 37_{25}\\
&&&71 & 13 & 13_{36}&&&&&&\\
&&&71 & 13 & 13_{35}&&&&&&\\
\hline
\end{array}\]

\[\begin{array}{|c|c|c|}\hline
x & g^6x & v_2^{20}h_1x\\\hline
15_0&39_{0}&36_{40}\\\hline
14_{20}&0&0\\
14_{19}&38_{19}&35_{59}\\\hline
13_{41}&0&0\\
13_{40}&0&0\\
13_{39}&37_{29}&34_{69}\\
13_{38}&37_{28}&34_{68}\\
13_{37}&37_{27}&34_{67}\\
13_{36}&37_{26}&34_{66}\\
13_{35}&37_{25}&34_{65}\\
\hline
\end{array}\]
\end{table}

\begin{table}[H] 
\caption{$Ext_{A(2)}^{s-1,t-8}(\A\sma D\A\sma bo_1)$}\label{ADA1-10bo1}
\[
\begin{array}{|c|c|c||c|c|c||||c|c|c||c|c|c|}
\hline
t-s&s&s_g&t-s&s&s_g&t-s&s&s_g&t-s&s&s_g\\
\hline
\hline
70&15&&71&15&&190&39&&191&39&\\
\hline
70&14&&71&14&&190&38&&191&38&\\
\hline
70&13 & 12_{9} &71 & 13 & 12_{17}&190 & 37 & 36_{9} &191 & 37 & 36_{17}\\
70&13 & 12_{8} &71 & 13 & 12_{16}&190 & 37 & 36_{8} &191 & 37 & 36_{16}\\
70&13 & 12_{7} &71 & 13 & 12_{15}&190 & 37 & 36_{7} &191 & 37 & 36_{15}\\
70&13 & 12_{6} &71 & 13 & 12_{14}&190 & 37 & 36_{6} &191 & 37 & 36_{14}\\
70&13 & 12_{5} &71 & 13 & 12_{13}&190 & 37 & 36_{5} &191 & 37 & 36_{13}\\
70&13 & 12_{4} &71 & 13 & 12_{12}&190 & 37 & 36_{4} &191 & 37 & 36_{12}\\
&&&71 & 13 & 12_{11}&&&&191 & 37 & 36_{11}\\
&&&71 & 13 & 12_{10}&&&&191 & 37 & 36_{10}\\
\hline
70 & 12 & 11_{39} &71 & 12 & 11_{45}&190&36&&191&36&\\
70 & 12 & 11_{38} &71 & 12 & 11_{44}&&&&&&\\
70 & 12 & 11_{37} &71 & 12 & 11_{43}&&&&&&\\
70 & 12 & 11_{36} &71 & 12 & 11_{42}&&&&&&\\
70 & 12 & 11_{35} &71 & 12 & 11_{41}&&&&&&\\
70 & 12 & 11_{34} &71 & 12 & 11_{40}&&&&&&\\
70 & 12 & 11_{33} &&&&&&&&&\\
\hline
\end{array}\]
\[\begin{array}{|c|c|c|}\hline
x&g^6x&v_2^{20}h_1x\\\hline
12_{17}&36_{17}&33_{73}\\
12_{16}&36_{16}&33_{72}\\
12_{15}&36_{15}&33_{71}\\
12_{14}&36_{14}&33_{70}\\
12_{13}&36_{13}&33_{69}\\
12_{12}&36_{12}&33_{68}+33_{73}\\
12_{11}&36_{11}&33_{67}+33_{73}\\
12_{10}&36_{10}&33_{66}+33_{71}+33_{72}\\\hline
\end{array}\]
\end{table}
\newpage
\begin{table}[H] 
\caption{$Ext_{A(2)}^{s-1,t-16}(\A\sma D\A\sma bo_2)$}\label{ADA1-10bo2}
\[
\begin{array}{|c|c|c||c|c|c||||c|c|c||c|c|c|}
\hline
t-s&s&s_g&t-s&s&s_g&t-s&s&s_g&t-s&s&s_g\\
\hline
\hline
70&15&&71&15&&190&39&&191&39&\\
\hline
70&14&&71&14&&190&38&&191&38&\\
\hline
70&13&&71&13&&190&37&&191&37&\\
\hline
70&12&&71&12&&190&36&&191&36&\\\hline
\end{array}\]
\end{table}

\begin{table}[H] 
\caption{$Ext_{A(2)}^{s-2,t-16}(\A\sma D\A\sma bo_1\sma bo_1)$}\label{ADA1-10bo1bo1}
\[
\begin{array}{|c|c|c||c|c|c||||c|c|c||c|c|c|}
\hline
t-s&s&s_g&t-s&s&s_g&t-s&s&s_g&t-s&s&s_g\\
\hline
\hline
70&15&&71&15&&190&39&&191&39&\\
\hline
70&14&&71&14&&190&38&&191&38&\\
\hline
70&13&&71&13&&190&37&&191&37&\\
\hline
70&12&&71&12&10_1&190&36&&191&36&34_1\\
&&&71&12&10_0&&&&191&36&34_0\\\hline
\end{array}\]
\[\begin{array}{|c|c|c|}\hline
x&g^6x&v_2^{20}h_1x\\\hline
10_{1}&34_1&31_{81}\\
10_{0}&34_0&31_{80}\\\hline
\end{array}\]
\end{table}
\clearpage


\end{document}